\documentclass[11pt]{article}

\usepackage{amsmath, amsfonts, amssymb, amsthm,  graphicx, enumerate, dsfont}
\usepackage[letterpaper, left=1truein, right=1truein, top = 1truein, bottom = 1truein]{geometry}

\usepackage[numbers, square]{natbib}

\usepackage[%
            CJKbookmarks=true,
            bookmarksnumbered=true,
            bookmarksopen=true,
            colorlinks=true,
            citecolor=red,
            linkcolor=blue,
            anchorcolor=red,
            urlcolor=blue,
            ]{hyperref}

\usepackage[usenames,dvipsnames]{color}

\usepackage[ruled, vlined, lined, commentsnumbered]{algorithm2e}

\usepackage{prettyref,soul,xspace}

\usepackage{tikz}
\usepackage{pgfplots,makecell}

\theoremstyle{remark}

\theoremstyle{plain}

\newtheorem{lemma}{Lemma}
\newtheorem{theorem}{Theorem}
\newtheorem{proposition}{Proposition}

\newcommand{\p}{\mathbb{P}}
\newcommand{\E}{\mathbb{E}}

\newcommand{\iid}{\stackrel{iid}{\sim}}

\newcommand{\br}[1]{\left( #1 \right)}

\newcommand{\cbr}[1]{\left\{ #1 \right\}}
\newcommand{\pbr}[1]{\p\left( #1 \right)}
\newcommand{\ebr}[1]{\exp\left( #1 \right)}
\newcommand{\abs}[1]{\left| #1 \right|}

\newcommand{\mathr}{\mathbb{R}}

\newcommand{\indic}[1]{{\mathbb{I}\left\{{#1}\right\}}}
\newcommand{\iprod}[2]{\left \langle #1, #2 \right\rangle}
\newcommand{\norm}[1]{\left\|{#1} \right\|}

\newcommand{\fnorm}[1]{\norm{#1}_{\rm F}}
\newcommand{\argmin}{\mathop{\rm argmin}}
\newcommand{\argmax}{\mathop{\rm argmax}}
\newcommand{\normf}[1]{\|{#1}\|_{\rm F}}

\newcommand{\im}{{\rm Im}}
\newcommand{\re}{{\rm Re}}
\newcommand{\one}{\mathds{1}}
\def\H{{ \mathrm{\scriptscriptstyle H} }}
\def\T{{ \mathrm{\scriptscriptstyle T} }}
\newcommand{\mathc}{\mathbb{C}}
\newcommand{\maths}{\mathbb{S}}
\newcommand{\mathx}{\mathcal{X}}

\newcommand{\mathw}{\mathcal{W}}
\newcommand{\matho}{\mathcal{O}}

\newcommand{\mathp}{\mathcal{P}}

\newcommand{\ellod}{\ell^\text{od}}
\renewcommand{\tilde}{\widetilde}
\renewcommand{\hat}{\widehat}

\DeclareFontFamily{U}{mathx}{}
\DeclareFontShape{U}{mathx}{m}{n}{<-> mathx10}{}
\DeclareSymbolFont{mathx}{U}{mathx}{m}{n}
\DeclareMathAccent{\widecheck}{0}{mathx}{"71}
\renewcommand{\check}{\widecheck}

\title{Exact Minimax Optimality of Spectral Methods in Phase Synchronization and Orthogonal Group Synchronization}

\author{Anderson Ye Zhang\\
~\\
University of Pennsylvania
}

\begin{document}
\maketitle

\begin{abstract}
We study the performance of the spectral method for  the phase synchronization problem with additive Gaussian noises and incomplete data. The spectral method utilizes the leading eigenvector of the data matrix followed by a normalization step. We prove that it achieves the minimax lower bound of the problem with a matching leading constant under a squared $\ell_2$ loss. This shows that the spectral method has the same performance as more sophisticated procedures including maximum likelihood estimation, generalized power method,  and semidefinite programming, as long as consistent parameter estimation is possible.
To establish our result, we first have a novel choice of the population eigenvector, which enables us to establish the exact recovery of the spectral method when there is no additive noise. 
We then develop a new perturbation analysis toolkit for the leading eigenvector and show  it can be well-approximated by its first-order approximation with a small $\ell_2$ error. We further extend our analysis to establish the exact minimax optimality of the spectral method for the orthogonal group synchronization.
\end{abstract}

\section{Introduction}\label{sec:intro}
We consider the phase synchronization problem with additive Gaussian noises and incomplete data \cite{abbe2017group, bandeira2017tightness, zhong2018near, gao2021exact}. Let $z^*_1,\ldots,z^*_n \in\mathbb{C}_1$ where $\mathbb{C}_1:= \{x\in\mathbb{C}:\abs{z}=1\}$, the set of all unit complex numbers. Then each $z^*_j$ can be written equivalently as $e^{i\theta^*_j}$ for some phase (or angle) $\theta^*_j\in[0,2\pi)$ . For each $1\leq j<k\leq n$, the observation  $X_{jk}\in\mathbb{C}$ is missing at random. Let $A_{jk}\in\{0,1\}$ and $X_{jk}$ satisfy
\begin{align}\label{eqn:model}
X_{jk} :=
\begin{cases}
{z_j^*\overline{z_k^*} + \sigma W_{jk}},\text{ if }A_{jk}=1,\\
0,\quad\quad\quad\quad\quad\text{ if }A_{jk}=0,
\end{cases}
\end{align}
where $A_{jk}\sim\text{Bernoulli}(p)$ and $W_{jk}\sim \mathcal{CN}(0,1)$. That is, each $X_{jk}$ is missing with probability $1-p$ and is denoted as 0. If it is not missing, it is equal to $z_j^*\overline{z_k^*}$ with an additive noise $\sigma W_{jk}$ where $W_{jk}$ follows the standard complex Gaussian distribution: $\re(W_{jk}),\im(W_{jk})\sim \mathcal{N}(0,1/2)$ independently. Each $A_{jk}$ is the indicator of whether $X_{jk}$ is observed or not.
We assume all random variables $\{A_{jk}\}_{1\leq j<k\leq n},\{W_{j,k}\}_{1\leq j<k\leq n}$ are independent of each other. The goal is to estimate the phase vector $z^* := (z_1^*,\ldots, z_n^*)\in\mathbb{C}_1^n$ from $\{A_{jk}\}_{1\leq j<k\leq n}$ and $\{X_{jk}\}_{1\leq j<k\leq n}$.

The observations can be seen as entries of a matrix $X\in\mathbb{C}^{n\times n}$ with $X_{jj}:=0$ and $X_{kj}:=\overline{X_{jk}}$ for any $1\leq j < k\leq n$. Define $A_{jj}:=0$ and $A_{kj}:=A_{jk}$ for  all $1\leq j<k\leq n$. Then the matrix $A\in\{0,1\}^{n\times n}$ can be interpreted as the adjacency matrix of an Erd\"{o}s-R\'enyi random graph with edge probability $p$. Define $W\in\mathc^{n\times n}$ such that $W_{jj}:=0$ and $W_{kj}:=\overline{W_{jk}}$ for  all $1\leq j<k\leq n$. Then all the matrices $A,W,X$ are Hermitian and $X$ can be written equivalently as
\begin{align}\label{eqn:X_matrix_form}
X= A\circ\br{z^*z^{*\H} + \sigma W} = A\circ\br{z^*z^{*\H}}+ \sigma A \circ W.
\end{align}
Note that $X$ can be seen as a noisy version of
\begin{align}\label{eqn:EX}
\E X = pz^*z^{*\H}-pI_n
\end{align}
whose leading eigenvector is $z^*/\sqrt{n}$.
This motivates the following spectral method \cite{singer2011angular, gao2021exact, cucuringu2016sync}. Let $u\in\mathbb{C}^n$  be the leading eigenvector of $X$. Then the spectral estimator $\hat z\in\mathbb{C}_1^n$ is defined as
\begin{align}\label{eqn:spectral_estimator_def}
\hat z_j := \begin{cases}
\frac{u_j}{\abs{u_j}},\text{ if }u_j\neq 0,\\
1,\quad\text{ if }u_j=0,
\end{cases}
\end{align}
for each $j\in[n]$,  where each $u_j$ is normalized so that $\hat z_j\in\mathc_1$. The performance of the spectral estimator can be quantified by a normalized squared $\ell_2$ loss
\begin{align}\label{eqn:ell_def}
\ell(\hat z,z^*):= \min_{a\in\mathc_1} \frac{1}{n}\sum_{j=1}^n \abs{\hat z_j - z_j^*a}^2,
\end{align}
where the minimum over $\mathc_1$ is due to the  fact that $z^*_1,\ldots, z^*_n$ are identifiable only up to a phase.

The spectral estimator $\hat z$ is simple and easy to implement. Regarding its theoretical performance, it was suggested in \cite{gao2021exact} that an upper bound $\ell(\hat z,z^*)\leq C(\sigma^2+1)/(np)$ holds with high probability for some constant $C$. However,  the minimax risk of the phase synchronization was established in \cite{gao2021exact} and has the following lower bound:
\begin{align}\label{eqn:simple_minimax}
\inf_{z\in\mathc^n}\sup_{z^*\in \mathc_1^n}\E\ell(z,z^*)\geq \br{1-o(1)}\frac{\sigma^2}{2np}.
\end{align}
To provably achieve the minimax risk, the spectral method is often used as an initialization for some more sophisticated  procedures. For example, it was used to initialize a  generalized power method (GPM) \cite{boumal2016nonconvex, liu2017estimation, perry2018message} in \cite{gao2021exact}.
Nevertheless, numerically the performance of the spectral method is already very good and the improvement from GPM is often marginal. 
This raises the following questions about the performance of the spectral method: Can we derive a sharp upper bound? Does the spectral method already achieve the minimax risk or not?

In this paper, we provide complete answers to these questions. We carry out a sharp $\ell_2$ analysis of the  performance of the spectral estimator $\hat z$ and further show it achieves the minimax risk with the correct constant. Our main result is summarized below in Theorem \ref{thm:asymptotic} in asymptotic form.  Its non-asymptotic version will be given in Theorem \ref{thm:phase_main} that only requires $\frac{np}{\sigma^2}$, $\frac{np}{\log n}$ to be greater than a certain constant. We note that in this paper, $p$ and $\sigma^2$ are not constants but  functions of $n$. This dependence can be more explicitly represented as $p_n$ and $\sigma^2_n$. However, for simplicity of notation and readability, we choose to denote these as $p$ and $\sigma^2$ throughout the paper.

\begin{theorem}\label{thm:asymptotic}
Assume $\frac{np}{\sigma^2}\rightarrow\infty$ and $\frac{np}{\log n}\rightarrow\infty$. There exists some $\delta=o(1)$ such that with high probability,
\begin{align}\label{eqn:asymptotic}
\ell(\hat z,z^*) \leq (1+\delta)\frac{\sigma^2}{2np}.
\end{align}
As a consequence, when $\sigma=0$ (i.e., there is no additive noise), the spectral method recovers $z^*$ exactly (up to a phase) with high probability as long as $\frac{np}{\log n}\rightarrow\infty$.
\end{theorem}

 Theorem \ref{thm:asymptotic} shows that $\hat z$ is not only rate-optimal but also achieves the exact  minimax risk with the correct leading constant in front of the optimal rate. The conditions needed in Theorem \ref{thm:asymptotic} are necessary for consistent estimation of $z^*$ in the phase synchronization problem.
The condition $\frac{np}{\sigma^2}\rightarrow\infty$ is needed so that $z^*$ can be estimated with a vanishing error according to the  minimax lower bound (\ref{eqn:simple_minimax}). The condition $\frac{np}{\log n}\rightarrow\infty$ allows $p$ to decrease as $n$ grows and is close to the $\frac{np}{\log n}>(1+\epsilon)$ condition required for $A$ to be connected.
If $A$ has disjoint subgraphs, it is impossible to estimate $z^*$ with a global phase. These two conditions are also needed in \cite{gao2021exact, gao2022sdp} to establish the optimality of MLE (maximum likelihood estimation), GPM, and SDP (semidefinite programming). 
Under these two conditions, \cite{gao2021exact} used the spectral method as an initialization for the GPM and shows  GPM achieves the minimax risk after $\log(1/\sigma^2)$ iterations. On the contrary, Theorem \ref{thm:asymptotic} shows that the spectral method already achieves the minimax risk. 
This means that the spectral estimator is minimax optimal and achieves the correct leading constant whenever consistent estimation is possible, and in this parameter regime, it is as good as MLE, GPM, and SDP.

There are two key novel components toward establishing Theorem \ref{thm:asymptotic}. The first  is a new idea regarding the choice of the ``population eigenvector'' as $u$ can be viewed as its sample counterpart obtained from data. Due to (\ref{eqn:EX}) and the fact $pz^*z^{*\H}=np (z^*/\sqrt{n})(z^*/\sqrt{n})^\H$ is rank-one with the eigenvector $z^*/\sqrt{n}$, existing literature such as \cite{gao2021exact,gao2021optimal,ling2022near} treated $z^*/\sqrt{n}$ as the population eigenvector and study the perturbation of $u$ with respect to $z^*/\sqrt{n}$. This seems natural but turns out to be unappealing as it fails to explain why the spectral estimator is able to recover all phases exactly when $\sigma=0$, i.e., when there is no additive noise. Instead, we denote $\check u\in\mathr^{n}$ as the leading eigenvector of $A$ and regard $u^*\in\mathc^n$, defined as
\begin{align}\label{eqn:u_star_def}
u^*:= z^* \circ \check u,
\end{align}
i.e., $u^*_j = z^*_j \check u_j$ for each $j\in[n]$, as the population eigenvector.
Note that $u^*$ is random as it depends on the graph $A$. A careful analysis of $u^*$ reveals that it is the leading eigenvector of $A\circ z^*z^{*\H}$ (see Lemma \ref{lem:no_additive_noise}). In addition, Proposition \ref{prop:no_additive_noise} shows that with high probability,  $u^*_j/|u^*_j| = z^*_j$ for each $j\in[n]$, up to a global phase.  Since $u$ equals $u^*$ when $\sigma=0$, it successfully explains the exact recovery of the spectral method in the no-additive-noise case. Another advantage of viewing $u^*$  as the population eigenvector, instead of $z^*/\sqrt{n}$, is that intuitively $u^*$ is closer to $u$ than $z^*/\sqrt{n}$ is. This is because $A\circ z^*z^{*\H}$ is closer to the data matrix $X$ than $pz^*z^{*\H}$ is.

The second key component is a novel perturbation analysis for $u$. Classical matrix perturbation theory such as Davis-Kahan Theorem focuses on analyzing $\inf_{b\in\mathc_1}\|u-u^*b\|$. We go beyond  it and show $u$ can be well-approximated by its first-order approximation $\tilde u$ defined as
\begin{align}\label{eqn:tilde_u_def_intro}
\tilde u := \frac{Xu^*}{\norm{Xu^*}},
\end{align}
 in the sense that the difference between these two vectors (up to a phase) has a small $\ell_2$ norm. This means that when $np\gtrsim \log n$ and $np\gtrsim \sigma^2$, we have 
 \begin{align}\label{eqn:intro_decomposition}
 \inf_{b\in\mathc_1} \norm{u-\tilde ub}\lesssim \frac{\sigma^2 + \sigma}{np},
 \end{align}
 with high probability (see Proposition \ref{prop:eigenvector_perturbation}).
 In fact, our perturbation analysis extends beyond the phase synchronization problem. What we establish is a general perturbation theory that can be applied to  two arbitrary Hermitian matrices (see Lemma \ref{lem:eigenvector_perturbation}), which might be of independent interest.

With the help of these two key components, we then carry out an entrywise analysis for each $\hat z_j =u_j/\abs{u_j}$. 
Note that $u_j$ can be decomposed into $\tilde u_j$ and the difference between $u_j$ and $\tilde u_j$ (up to some global phase). We can decompose the error of $\hat z_j$ into two parts, one is related to the estimation error of  $\tilde u_j/|\tilde u_j|$, and the other is related to the magnitude of the difference between $u_j$ and $\tilde u_j$ (up to some global phase). Summing over all coordinates, 
the first part eventually leads to the minimax risk $(1+o(1)){\sigma^2}/{2np}$ and the second part is essentially negligible due to (\ref{eqn:intro_decomposition}),
which leads to the exact minimax optimality of the spectral estimator.

~\\
\indent\emph{Orthogonal Group Synchronization.} The above analysis for the phase synchronization can be extended to quantify the performance of the spectral method for orthogonal group synchronization,
which is about orthogonal matrices instead of phases. 
Let $d>0$ be an integer. Define
\begin{align}\label{eqn:od_set_def}
\mathcal{O}(d):=\cbr{U\in\mathr^{d\times d}:UU^\T=U^\T U=I_d}
\end{align}
to include all orthogonal matrices in $\mathr^{d\times d}$. Let $Z^*_1,\ldots,Z^*_n\in\mathcal{O}(d)$. Analogous to (\ref{eqn:model}), we consider the problem with additive Gaussian noise and incomplete data \cite{gao2021optimal, ling2022improved, ling2022near}.
For each $1\leq j<k\leq n$,
we observe $\mathx_{jk} :=
{Z_j^*  Z_k^{*\T} + \sigma \mathw_{jk}}\in\mathr^{d\times d}$ when $A_{jk}=1$, where $ \mathw_{jk}$ follows the standard matrix Gaussian distribution.  The goal is to recover $Z_1^*,\ldots,Z_n^*$  from  $\{\mathx_{jk}\}_{1\leq j<k\leq n}$ and $\{A_{j,k}\}_{1\leq j<k\leq n}$. This is known as the  orthogonal group synchronization (or $\mathcal{O}(d)$ synchronization).

The observations $\{\mathx_{jk}\}_{1\leq j<k\leq n}$ can be seen as submatrices of an $nd\times nd$  matrix $\mathx$ with $\mathx_{jj}:=0$ and $\mathx_{kj}:=\mathx_{jk}^\T$ for any $1\leq j<k\leq n$. Then $\mathx$ is symmetric and can be seen as a noisy version of
\begin{align}\label{eqn:intro_EX_od}
\E \mathx = p Z^*Z^{*\T} - pI_{nd}
\end{align}
whose leading eigenspace is $Z^*/\sqrt{n}$,
where $Z^*\in \matho(d)^{n}$ is an $nd\times d$ matrix such that its $j$th submatrix is $Z^*_j$. Similar to the phase synchronization, we have the following spectral method. Let $\lambda_1\geq \ldots\geq \lambda_d$ be the largest $d$ eigenvalues of $\mathx$ and $u_1,\ldots,u_d\in\mathr^{nd}$ be their corresponding eigenvectors. Denote $U:=(u_1,\ldots,u_d)\in\mathr^{nd\times d}$ as the eigenspace that includes the top $d$ eigenvectors of $\mathx$. For each $j\in[n]$, denote $U_j\in\mathr^{d\times d}$ as its $j$th submatrix. Then the spectral estimator  $\hat Z_j\in\matho(d)$ is defined as
\begin{align}\label{eqn:intro_hat_Z_def}
\hat Z_j :=\begin{cases}
\mathp(U_j),\text{ if }\det(U_j)\neq 0,\\
I_d,\quad\;\;\;\text{ if }\det(U_j)=0,
\end{cases}
\end{align}
for each $j\in[n]$.
Here the mapping $\mathp:\mathr^{d\times d}\rightarrow \matho(d)$ is derived from the  polar decomposition and serves as a normalization step for each $U_j$ such that $\hat Z_j\in\matho(d)$. Let $\hat Z\in\matho(d)^n$ be an $nd\times d$ matrix such that $\hat Z_j$  is its $j$th submatrix for each $j\in[n]$.  
Then the performance of $\hat Z$ can be quantified by a loss function $\ell^\text{od}(\hat Z, Z^*)$ that is analogous to (\ref{eqn:ell_def}). The detailed definitions of $\mathp$ and $\ell^\text{od}$  are deferred to Section \ref{sec:od}.

The spectral method $\hat Z$ was used as an initialization in \cite{gao2021optimal} for
a variant of GPM to achieve the exact minimax risk $(1+o(1))\frac{d(d-1)\sigma^2}{2np}$ for $d=O(1)$. To conduct a sharp analysis of its statistical performance, we extend our novel perturbation analysis from analyzing the leading eigenvector to the leading eigenspace. Recall $\check u$ is the leading eigenvector of $A$. Analogous to (\ref{eqn:u_star_def}),
we have a novel choice of the population eigenspace $U^*$, defined as
\begin{align}\label{eqn:U_star_def}
U^*:= Z^* \circ (\check u \otimes \one_d),
\end{align}
and view $U$ as its sample counterpart. This is different from existing literature \cite{gao2021optimal, zhu2021orthogonal} which uses $Z^*/\sqrt{n}$ as the population eigenspace. Our choice of $U^*$  enables the establishment of the exact recovery of the spectral method when there is no additive noise (i.e., $\sigma=0$), as seen Proposition \ref{prop:no_additive_noise_od}, and is closer to $U$ than $Z^*/\sqrt{n}$ is.

 The first-order approximation of $U$ is a matrix determined by $\mathx U^*$ whose explicit expression will be given later in (\ref{eqn:tilde_U_def}).
We then show $U$ can be well-approximated by its first-order approximation, analogous to (\ref{eqn:tilde_u_def_intro}),
 with a remainder term of a small $\ell_2$ norm (see Proposition \ref{prop:eigenspace_perturbation}). This is a consequence of a more general eigenspace perturbation theory (see Lemma \ref{lem:eigenspace_perturbation}) for two arbitrary Hermitian matrices.
Using the first-order approximation, we then carry out an entrywise analysis for $\hat Z$.
 Our main result for the spectral method in the $\matho(d)$ synchronization is summarized below in Theorem \ref{thm:od_asymptotic}. The non-asymptotic version will be given in Theorem \ref{thm:od}.

\begin{theorem}\label{thm:od_asymptotic}
Assume $\frac{np}{\sigma^2}\rightarrow\infty$, $\frac{np}{\log n}\rightarrow\infty$, and $2\leq d=O(1)$. There exists some $\delta=o(1)$ such that with high probability,
\begin{align*}
\ell^\text{od}(\hat Z,Z^*) \leq (1+\delta)\frac{d(d-1)\sigma^2}{2np}.
\end{align*}
As a consequence, when $\sigma=0$ (i.e., there is no additive noise), the spectral method recovers $Z^*$ exactly (up to an orthogonal matrix) with high probability as long as $\frac{np}{\log n}\rightarrow\infty$.
\end{theorem}

Theorem \ref{thm:od_asymptotic} shows that the spectral method $\hat Z$ achieves exact minimax optimality as it matches the minimax lower bound $(1+o(1))\frac{d(d-1)\sigma^2}{2np}$ established in \cite{gao2021optimal}. Similar to the phase synchronization, the two conditions needed in Theorem \ref{thm:od_asymptotic} so that consistent estimation of $Z^*$ is possible. They are also needed in  \cite{gao2021optimal}  to achieve the minimax risk by a variant of GPM initialized by the spectral method. On the contrary, Theorem \ref{thm:od_asymptotic} shows that in this parameter regime, the spectral method is already minimax optimal with the correct leading constant.

~\\
\indent\emph{Related Literature.}
Synchronization is a fundamental problem in applied math and statistics.
Various methods have been studied for both phase synchronization and $\matho(d)$ synchronization, including the maximum likelihood estimation (MLE) \cite{gao2021exact, zhong2018near}, GPM \cite{zhong2018near,liu2017estimation, gao2021exact, gao2021optimal, ling2022improved, shen2016normalized}, SDP \cite{arie2012global,singer2011three, ling2020solving, fan2021joint, wang2013exact, gao2022sdp, javanmard2016phase},  spectral methods \cite{arie2012global,singer2011three, romanov2020noise,boumal2013robust, perry2016optimality, filbir2020recovery}, and message passing \cite{perry2018message, lerman2022robust, shi2020message}. The theoretical performance of spectral methods was investigated in \cite{gao2021exact, gao2021optimal, ling2022near}  and crude error bounds under $\ell_2$ or Frobenius norm were obtained. An $\ell_\infty$-type error bound for  spectral methods was also given in \cite{ling2022near}.

Fine-grained perturbation analysis of eigenvectors has gained increasing attention in recent years for various low-rank matrix problems in machine learning and statistics. Existing literature has mostly focused on establishing $\ell_\infty$ bounds for eigenvectors \cite{abbe2020entrywise, chen2021spectral, fan2018eigenvector} or $\ell_{2,\infty}$ bounds on eigenspaces \cite{lei2019unified, cape2019two, cai2021subspace, agterberg2021entrywise}. For instance, \cite{abbe2020entrywise} developed $\ell_\infty$-type bounds for the difference between eigenvectors (or eigenspaces) and their first-order approximations. In this paper, we focus on developing sharp $\ell_2$-type perturbation bounds, where direct applications of existing $\ell_\infty$-type results will result in extra logarithm factors.

For the phase synchronization problem, \cite{filbir2020recovery, iwen2020phase, preskitt2018phase} investigated variants of spectral methods based on Laplacian matrices. Instead of using the leading eigenvector of $X$ as in this paper, they utilize  the eigenvector corresponding to the smallest eigenvalue of  $D-X$ or $I_n - D^{-\frac{1}{2}} X D^{-\frac{1}{2}}$, where $D\in\mathr^{n\times n}$ is the degree matrix of $A$ with diagonal entries $D_{jj}:	=\sum_{k\neq j} A_{jk}$ and off-diagonal entries set to zero. These studies have established upper bounds for the performance of their spectral methods applicable to general graphs $A$ and additive noise $W$. Our focus, however, is on Erd\"{o}s-R\'enyi random graphs with Gaussian noise. Under our setting, their results imply an upper bound of $\frac{C\sigma^2}{np}$, where $C$ is a constant significantly greater than 1. In contrast, our work establishes a sharp upper bound with the correct leading constant $1/2$. Our analytical approach could potentially be extended to their methods to  achieve the correct constant  $1/2$.

The existing literature \cite{gao2021exact, gao2022sdp, gao2021optimal} explored the exact minimax risk in synchronization problems, focusing primarily on minimax lower bounds and analyzing MLE, GPM, and SDP. While our study shares a thematic resemblance with these prior efforts, it fundamentally diverges in both analysis and proof techniques.  Previous studies hinge on contraction properties of the generalized power iteration (GPI), demonstrating the iterative reduction in GPM error until an optimal error is achieved. This approach further interprets MLE as a GPI fixed point and SDP as an extension of GPI in a higher-dimensional space, thereby establishing their optimality.
In contrast, this paper employs a novel strategy specifically tailored for the spectral method. Instead of relying on the GPI framework, which proves inadequate for spectral analysis, we introduce a new perturbation toolkit designed for eigenvector analysis. This toolkit provides precise characterization of eigenvector perturbation and leads to the optimality of the spectral method. It opens new avenues for research and application beyond synchronization problems.

~\\
\indent\emph{Organization.}
We study the phase synchronization in Section \ref{sec:phase}. We first establish the exact recovery of the spectral method in the no-additive-noise case in Section \ref{sec:no_additive}. Then in Section \ref{sec:phase_approximation}, we present our main technical tool for quantifying the distance between the leading eigenvector and its first-order approximation. We then carry out an entry-wise analysis of the spectral method and obtain non-asymptotic sharp upper bounds in Section \ref{sec:sharp_constant}. Finally, we consider the extension to  the orthogonal group synchronization in Section \ref{sec:od}. Proofs of results for the phase synchronization are given in Section \ref{sec:proof_phase}. Due to the page limit, we prove Lemma \ref{lem:eigenspace_perturbation} in Section \ref{sec:proof_od} and include the proofs of other results for the orthogonal group synchronization in the Appendix. 

~\\
\indent\emph{Notation.} For any positive integer $n$, we write $[n]:=\{1,2,\ldots, n\}$  and $\one_n := (1,1,\ldots, 1)\in\mathr^n$. Denote $I_n$ as the $n\times n$ identity matrix  and $J_n:=\one_n\one_n^\T\in\mathr^{n\times n}$ as the $n\times n$ matrix with all entries being one.
 Given $a,b\in\mathr$, we denote $a\wedge b:=\min\{a,b\}$ and $a\vee b:=\max\{a,b\}$. For a complex number $x\in\mathc$, we use $\bar x$ for its complex conjugate, $\re(x)$ for its real part, $\im(x)$ for its imaginary part,  and $|x|$ for its modulus.  
Denote $\maths_n:=\cbr{x\in\mathc^n:\norm{x}=1}$ as including all unit vectors in $\mathc^n$. For a complex vector $x=(x_j)\in\mathbb{C}^d$, we denote $\|x\|=({\sum_{j=1}^d |x_j|^2})^{1/2}$ as its Euclidean norm. 
 For a complex matrix $B =(B_{jk})\in\mathbb{C}^{d_1\times d_2}$, we  use $B^\H \in\mathbb{C}^{d_2\times d_1}$ for its conjugate transpose such that $B^{\H}_{jk}= \overline{B_{kj}}$. 
The Frobenius norm and the operator norm of $B$ are defined by $\fnorm{B}:=({\sum_{j=1}^{d_1}\sum_{k=1}^{d_2}|B_{jk}|^2})^{1/2}$ and $\norm{B} := \sup_{u\in\mathbb{C}^{d_1},v\in\mathbb{C}^{d_2}:\norm{u}=\norm{v}=1} u^\H Bv$. We use $\text{Tr}(B)$ for the trace of a squared matrix $B$. We denote $B_{j\cdot}$ as its $j$th row and define $\|B\|_{2\rightarrow\infty} := \max_{j\in[d_1]} \norm{B_{j\cdot}}$.
The notation $\det(\cdot)$ and $\otimes$ are used for determinant and Kronecker product.  For $U,V\in\mathbb{C}^{d_1\times d_2}$, $U\circ V\in\mathbb{R}^{d_1\times d_2}$ is the Hadamard product $U\circ V:=(U_{jk}V_{jk})$. For any $B\in\mathr^{d_1\times d_2}$, we denote $s_{\min}(B)$ as its smallest singular value.
For two positive sequences $\{a_n\}$ and $\{b_n\}$, $a_n\lesssim b_n$ and $a_n=O(b_n)$ both mean $a_n\leq Cb_n$ for some constant $C>0$ independent of $n$. We also write $a_n=o(b_n)$ or $\frac{b_n}{a_n}\rightarrow\infty$ when $\limsup_n\frac{a_n}{b_n}=0$.  
We use $\indic{\cdot}$ as the indicator function. Define $\matho(d_1,d_2):=\{V\in\mathr^{d_1\times d_2} : V^\T V=I_{d_2}\}$ to  include all $d_1\times d_2$ matrices  that have orthonormal columns.

\section{Phase Synchronization}\label{sec:phase}

\subsection{No-additive-noise Case}\label{sec:no_additive}
We first study a special case where there is  no additive noise (i.e., $\sigma=0$). In this setting, the data matrix $X = A\circ z^*z^{*\H}$.
Despite the data still being missing at random, we are going to show the spectral method is able to recover $z^*$ exactly, up to a phase.

Recall that  $\check u$ is the leading eigenvector of $A$ and $u^*$ is defined in (\ref{eqn:u_star_def}).  The following lemma points out the connection between $u^*$  and $A\circ z^*z^{*\H}$ as well as the connection between eigenvalues of $A$ and those of $A\circ z^*z^{*\H}$.
\begin{lemma}\label{lem:no_additive_noise}
The unit vector $u^*$ is the leading eigenvector of $A\circ z^*z^{*\H}$. 
That is, with $\lambda^*$ denoting as the largest eigenvalue of $A\circ z^*z^{*\H}$, we have
\begin{align}\label{eqn:u_star}
\br{A\circ z^*z^{*\H}} u^* = \lambda^* u^*,
\end{align}
In addition, all the eigenvalues of $A$ are also eigenvalues of $A\circ z^*z^{*\H}$, and vice versa.
\end{lemma}

Since $X=A\circ z^*z^{*\H}$ in the no-additive-noise case, we have $u=u^*$.
Note that $\hat z_j = u_j/|u_j|=u^*_j/|u^*_j|=z^*_j \check u_j  /|z^*_j \check u_j |$ for each $j\in[n]$. If $\check u_j>0$,  we have $\hat z_j = z_j^* $. If $\hat u_j<0$ instead, then $\hat z_j = -z_j^* $. If all the coordinates of $\check u$ are positive (or negative), we have $\hat z$ being equal to $z^*$ (or $-z^*$) exactly. The following proposition provides an $\ell_\infty$ control for the difference between $\check u$ and $\one_n/\sqrt{n}$, which are eigenvectors of $A$ and $\E A$, respectively. The proof of (\ref{eqn:no_additive_noise}) follows proofs of results in \cite{abbe2020entrywise}. When the right-hand side of (\ref{eqn:no_additive_noise}) is smaller than $1/\sqrt{n}$, it immediately establishes the exact recovery of $\hat z$.

\begin{proposition}\label{prop:no_additive_noise}
There exist some constants $C_1,C_2>0$ such that if $\frac{np}{\log n}>C_1$, we have
\begin{align}\label{eqn:no_additive_noise}
\min_{b\in\{1,-1\}}\max_{j\in[n]}\abs{\check u_j - \frac{1}{\sqrt{n}}b}\leq C_2\br{\sqrt{\frac{\log n}{np}} + \frac{1}{\log (np)}}\frac{1}{\sqrt{n}},
\end{align}
with probability at least $1-8n^{-10}$.
As a result, if $\frac{np}{\log n}> \max\cbr{C_1, 2C_2^2}$, we have $\ell(\hat z,z^*)=0$ with probability at least $1-8n^{-10}$.
\end{proposition}

Lemma \ref{lem:no_additive_noise} and Proposition \ref{prop:no_additive_noise} together establish the exact recovery of $\hat z$ for the special case where $\sigma=0$, through studying $u^*$. This provides a starting point for our analysis of the general case where $\sigma\neq 0$. From (\ref{eqn:X_matrix_form}), the data matrix $X$ is a noisy version of $A\circ z^*z^{*\H}$ with additive noise $\sigma A\circ W$ that scales with $\sigma$. As a result, in the following sections, we view $u^*$ as the population eigenvector and $u$ as its sample counterpart, studying the performance of the spectral method.

\subsection{First-order Approximation of The Leading Eigenvector}\label{sec:phase_approximation}

In this section, we provide a fine-grained perturbation analysis for the eigenvector $u$. Classical matrix perturbation theory, such as Davis-Kahan Theorem, can only give a crude upper bound for $\inf_{b\in\mathc_1}\norm{u-u^*b}$, which turns out to be  insufficient to derive a sharp bound for $\ell(\hat z,z^*)$. Instead, we develop a more powerful tool for perturbation analysis of $u$ using its first-order approximation $\tilde u$ defined in (\ref{eqn:tilde_u_def_intro}).
In fact, our tool goes beyond the phase synchronization problem and can be applied to arbitrary Hermitian matrices.

\begin{lemma}\label{lem:eigenvector_perturbation}
Consider  two Hermitian matrices $Y,Y^*\in\mathc^{n\times n}$. Let $\mu^*_1\geq\mu^*_2\geq \ldots \geq \mu^*_n$ be the eigenvalues of  $Y^*$. Let $v^*$ (resp. $v$) be the eigenvector of $Y^*$ (resp. $Y$) corresponding to its largest eigenvalue.
If $\norm{Y -Y^*}\leq \min\{\mu_1^*-\mu_2^*,\mu_1^*\}/4$, we have
\begin{align*}
\inf_{b\in\mathc_1}\norm{v-\frac{Yv^*}{\norm{Yv^*}}b} &\leq \frac{40\sqrt{2}}{9(\mu_1^*-\mu_2^*)}\Bigg(\br{\frac{4}{\mu_1^{*}-\mu_2^*} + \frac{2}{\mu_1^{*}} } \norm{Y-Y^*}^2 \\
&\quad\quad\quad\quad\quad\quad\quad  +\frac{\max\{|\mu_2^*|,|\mu_n^*|\}}{\mu_1^*}\norm{Y-Y^*}\Bigg).
\end{align*}
\end{lemma}

In Lemma \ref{lem:eigenvector_perturbation}, there are two matrices $Y,Y^*$ whose leading eigenvectors are $v,v^*$ respectively. It studies the $\ell_2$ difference between $v$ and $Yv^*/{\norm{Yv^*}}$ up to a phase. 
Let $\mu_1$ be the largest eigenvalue of $Y$.
The unit vector $Yv^*/{\norm{Yv^*}}$ is interpreted as the  first-order approximation of $v$, as $v$ can be decomposed into $v= Yv /\mu_1=  Yv^*/\mu_1+  Y(v-v^*)/\mu_1$ where the first term $ Yv^*/\mu_1$ is proportional to $Yv^*/{\norm{Yv^*}}$. If $Y^*$ is rank-one, meaning $\mu_2^*=\mu^*_n=0$, the upper bound in Lemma \ref{lem:eigenvector_perturbation} becomes ${80\sqrt{2} \norm{Y-Y^*}^2}/{(3\mu_1^{*2})}$. 
 Lemma \ref{lem:eigenvector_perturbation} itself might be of independent interest and be useful in other low-rank matrix problems.

The key to Lemma \ref{lem:eigenvector_perturbation}  is the following equation. Since $\mu_1 v =Yv$ and $\norm{Yv^*} \frac{Yv^*}{\norm{Yv^*} } =Yv^*$, we can derive (see (\ref{eqn:proof_1}) in the proof of Lemma \ref{lem:eigenvector_perturbation}):
\begin{align*}
\mu_1^{-1}\norm{Yv^*} (\mu_1 I_n-Y)\br{v- \frac{Yv^*}{\norm{Yv^*}}}= Y(\mu_1^{-1}Yv^* - v^*).
\end{align*}
Its left-hand side can be shown to be related to $\inf_{b\in\mathc_1}\norm{v -  {Yv^*b}/{\norm{Yv^*}}}$. By carefully studying and upper bounding its right-hand side, which does not involve $v$, we derive Lemma \ref{lem:eigenvector_perturbation}. 

Lemma \ref{lem:eigenvector_perturbation} requires the perturbation between $Y$ and $Y^*$ is not only small compared to the eigengap $\mu^*_1-\mu^*_2$, but also small compared to the leading eigenvalue $\mu^*_1$. A similar requirement is also needed in \cite{abbe2020entrywise} to establish $\ell_\infty$ bounds for the difference between the eigenvector and its first-order approximation. In contrast, classical theory such as Davis-Kahan theorem (see Lemma \ref{lem:davis}) only needs the perturbation to be small compared to the eigengap to bound $\inf_{b\in\mathc_1}\norm{v-v^*b}$. A natural question is whether the bound in Lemma \ref{lem:eigenvector_perturbation} can be modified to depend on eigenvalues only through the eigengaps.
It turns out this is not feasible, as it deals with the distance between $v$ and its first-order approximation $Yv^*/{\norm{Yv^*}}$, not the distance between $v$ and $v^*$ as in Davis-Kahan theorem. To illustrate it, consider the following counterexample. Let $e_1,\ldots, e_n$ be the canonical basis of $\mathr^n$. Let $\delta>0$. Define
\begin{align}\label{eqn:rev_10}
Y^*:= \text{diag}(0,-1,-1,\ldots, -1)\in\mathr^{n\times n},\text{ and }Y:= Y^* + \delta (e_1+ e_2)(e_1+e_2)^\T/2.
\end{align}
Then $\mu_1^*=0,\mu_2^*=-1, \mu^*_1 - \mu^*_2 =1, v^*=e_1, \norm{Y-Y^*} = \delta$, and $Yv^*/{\norm{Yv^*}}=(e_1+e_2)/\sqrt{2}$.  We can show $v$ has the following explicit expression (see Appendix \ref{sec:v_calculation}  for detailed calculation):
\begin{align}\label{eqn:rev_11}
v = \sqrt{\frac{1}{2}\br{1+ \frac{1}{\sqrt{1+\delta^2}}}} e_1+ \sqrt{\frac{1}{2}\br{1- \frac{1}{\sqrt{1+\delta^2}}}}e_2.
\end{align}
When $\delta$ is sufficiently close to 0, we have $v\approx v^*$. This is not surprising as it is consistent with the bound from Davis-Kahan theorem as the ratio between the perturbation and eigengap is $\norm{Y-Y^*}/(\mu^*_1 - \mu^*_2) = \delta\approx 0$. On the other hand, $\norm{v - Yv^*/{\norm{Yv^*}}}\approx \norm{e_1 - (e_1+e_2)/\sqrt{2}}=2-\sqrt{2}>0$ no matter how small $\delta$ may be. As a result, in this counterexample, $Yv^*/{\norm{Yv^*}}$ is not a good approximation of $v$ despite the sufficiently small perturbation.

Applying Lemma \ref{lem:eigenvector_perturbation} to the phase synchronization, we have the following result.
\begin{proposition}\label{prop:eigenvector_perturbation}
There exist constants $C_1,C_2,C_3>0$ such that if $\frac{np}{\log n}>C_1$ and $\frac{np}{\sigma^2}>C_2$, we have
\begin{align*}
\inf_{b\in\mathc_1} \norm{u-\tilde ub}\leq C_3\frac{\sigma^2+\sigma}{np},
\end{align*}
with probability at least $1-3n^{-10}$.
\end{proposition}

Proposition \ref{prop:eigenvector_perturbation} shows that $u$ is well-approximated by its first-order approximation $\tilde u$ (up to a phase) with an approximation error that is at most in the order of $(\sigma^2+\sigma)/np$. Note that we can show $\inf_{b\in\mathc_1}\norm{u-u^*b}$ is of order $\sigma/\sqrt{np}$ by using Davis-Kahan Theorem. This is much larger than the upper bound derived in Proposition \ref{prop:eigenvector_perturbation}, particularly when $np/\sigma^2$ is large. As a result, $\tilde u$ provides a precise characterization of $u$ with negligible $\ell_2$ error.

\subsection{Sharp $\ell_2$ Analysis of  The Spectral Estimator}\label{sec:sharp_constant}
In this section, we will conduct a sharp analysis of the performance of the spectral estimator $\hat z$ using the first-order approximation $\tilde u$ of the eigenvector $u$. According to Proposition \ref{prop:eigenvector_perturbation}, $u$ is close to $\tilde u$ (up to a phase) with a small difference. Then intuitively, $\hat z$ should be close to its counterpart that uses $\tilde u$ instead of $u$ in (\ref{eqn:spectral_estimator_def}), up to a global phase. For each $j\in[n]$, the distance of $\tilde u_j/\abs{\tilde u_j}$ from $z^*_j$ is essentially determined by $\overline{z^*_j}\tilde u_j$.
By the definition in (\ref{eqn:tilde_u_def_intro}), $\tilde u_j$ is proportional to $[Xu^*]_j$, the $j$th coordinate of $Xu^*$. With (\ref{eqn:X_matrix_form}), it leads to $\overline{z^*_j}\tilde u_j \propto \lambda^* \overline{z^*_j}u^*_j + \sigma \sum_{k\neq j}A_{jk}W_{jk} \overline{z^*_j}u_k^*$. Here the  first term $\lambda^* \overline{z^*_j}u^*_j$ can be interpreted as the signal as it is related to the population quantity $u^*_j$, which gives the exact recovery of the spectral method in the no-additive-noise case in Proposition \ref{prop:no_additive_noise}. As $u^*$ is close to $z^*/\sqrt{n}$, the second term is approximately equal to $n^{-1/2}\sum_{k\neq j}A_{jk}W_{jk} \overline{z^*_j}z_k^*$. Its contribution toward the estimation error is essentially determined by its imaginary part $n^{-1/2}\im(\sum_{k\neq j}A_{jk}W_{jk} \overline{z^*_j}z_k^*)$, which can be interpreted as the main error term. Summing over all $j\in[n]$,  the signals and the main error terms together lead to the minimax risk $\sigma^2/(2np)$. At the same time, contributions of approximation errors such as $\inf_{b\in\mathc_1} \norm{u-\tilde ub}$ turn out to be negligible. This leads to the following theorem on the performance of the spectral estimator.

\begin{theorem}\label{thm:phase_main}
There exist constants $C_1,C_2,C_3>0$ such that if $\frac{np}{\log n}>C_1$ and $\frac{np}{\sigma^2}>C_2$, we have
\begin{align*}
\ell(\hat z,z^*)\leq \br{1+C_3\br{\br{\frac{\sigma^2}{np}}^\frac{1}{4} + \sqrt{\frac{\log n}{np}}  +\frac{1}{\log (np)}}}\frac{\sigma^2}{2np},
\end{align*}
with probability at least $1-n^{-9} -\ebr{-\frac{1}{32}\br{\frac{np}{\sigma^2}}^\frac{1}{4}}$.
\end{theorem}

Theorem \ref{thm:phase_main} is non-asymptotic and its asymptotic version is presented in Theorem \ref{thm:asymptotic}.  It covers the no-additive-noise case (i.e., Proposition \ref{prop:no_additive_noise}), as it implies that $\ell(\hat z,z^*)=0$ with high probability when $\sigma=0$. Theorem \ref{thm:phase_main} shows that $\ell(\hat z,z^*)$ is equal to $\sigma^2/(2np)$ up to a factor that is determined by $(\sigma^2/(np))^{1/4}$, $\sqrt{\log n/(np)}$, and $1/\log (np)$. The first term is related to various approximation errors including the one from Proposition \ref{prop:eigenvector_perturbation}. The second and third terms are derived from (\ref{eqn:no_additive_noise}).

We can make a comparison between Theorem \ref{thm:phase_main} and the existing result $\ell(\hat z,z^*)\lesssim (\sigma^2+1)/np$ in \cite{gao2021exact}. There are two main improvements. First, we obtain the exact constant 1/2 for the error term $\frac{\sigma^2}{np}$, which gives a more accurate characterization of the performance of the spectral estimator.
Second, the $1/np$ error term in $(\sigma^2+1)/np$ no longer exists in Theorem \ref{thm:phase_main}.
We further compare Theorem \ref{thm:phase_main} with the minimax lower bound for the phase synchronization problem. The paper  \cite{gao2021exact} proved that  there exist constants $C_4,C_5>0$ such that if $\frac{np}{\sigma^2}\geq C_4$, we have
\begin{align}\label{eqn:minimax}
\inf_{z\in\mathc^n}\sup_{z^*\in \mathc_1^n}\E\ell(z,z^*)\geq \br{1-C_5\br{\frac{\sigma^2}{np}+\frac{1}{n}}}\frac{\sigma^2}{2np}.
\end{align}
Compared with (\ref{eqn:minimax}), the spectral estimator $\hat z$ is exact minimax optimal as it not only achieves the correct rate $\sigma^2/(np)$ but also the correct constant $1/2$. Under the parameter regime as in Theorem \ref{thm:phase_main},  \cite{gao2021exact, gao2022sdp} showed that MLE,  GPM (if properly initialized), and SDP achieve the exact minimax risk. Theorem \ref{thm:phase_main} points out that the spectral method is as good as these methods.

\section{Orthogonal Group Synchronization}\label{sec:od}
In this section, we will extend our analysis to matrix synchronizations where the quantities of interest are orthogonal matrices instead of phases. The orthogonal group synchronization problem has been briefly introduced in Section \ref{sec:intro}. Here we provide more details about the problem.

Let $d>0$ be an integer.  Recall the definition of $\mathcal{O}(d)$ in (\ref{eqn:od_set_def}) and that
$Z^*_1,\ldots,Z^*_n\in\mathcal{O}(d)$. For each $1\leq j<k\leq n$, the observation $\mathx_{jk}\in\mathr^{d\times d}$ is given by
\begin{align}\label{eqn:od}
\mathx_{jk} :=
\begin{cases}
{Z_j^*  Z_k^{*\T} + \sigma \mathw_{jk}},\text{ if }A_{jk}=1,\\
0,\quad\quad\quad\quad\quad\quad\text{ if }A_{jk}=0,
\end{cases}
\end{align}
where $A_{jk}\sim \text{Bernoulli}(p)$ and $\mathw_{jk}\sim \mathcal{MN}(0,I_d,I_d)$, i.e., the standard matrix Gaussian distribution\footnote{A random matrix $X$ follows a matrix Gaussian distribution $\mathcal{MN}(M,\Sigma,\Omega)$ if its density function is proportional to $\exp\left(-\frac{1}{2}\text{Tr}\left(\Omega^{-1}(X-M)^{\T}\Sigma^{-1}(X-M)\right)\right)$.}. We assume $\{A_{jk}\}_{1\leq j<k\leq n},\{\mathw_{j,k}\}_{1\leq j<k\leq n}$ are all independent of each other. Similar to the phase synchronization problem, the observations are missing at random with additive Gaussian noises. The goal is to recover $Z_1^*,\ldots,Z_n^*$  from  $\{\mathx_{jk}\}_{1\leq j<k\leq n}$ and $\{A_{j,k}\}_{1\leq j<k\leq n}$.

The data matrix $\mathx\in\mathr^{nd\times nd}$ can be written equivalently in a way that is analogous to (\ref{eqn:X_matrix_form}). Define $A_{jj}:=0$ and $A_{kj}:=A_{jk}$ for  all $1\leq j<k\leq n$. Define $\mathw\in\mathc^{nd\times nd}$ such that $\mathw_{jj}:=0_{d\times d}$ and $\mathw_{kj}:= \mathw_{jk}^\T$ for  all $1\leq j<k\leq n$.  Then we have the expression:
\begin{align}\label{eqn:mathX_matrix_form}
\mathx = (A\otimes J_d) \circ (Z^*Z^{*\T} + \sigma \mathw) = (A\otimes J_d)\circ Z^*Z^{*\T} +\sigma (A\otimes J_d) \circ \mathw.
\end{align}
From (\ref{eqn:intro_EX_od}), the data matrix $\mathx$ can be seen as a noisy version of $pZ^*Z^{*\T}$.
Since the columns of $Z^*$ are  orthogonal to each other, we have the following eigendecomposition: $p Z^*Z^{*\T} =np (Z^*/\sqrt{n})  (Z^*/\sqrt{n})^\T$ where $Z^*/\sqrt{n} \in \matho(nd,d)$. That is, $np$ is the only non-zero eigenvalue of $p Z^*Z^{*\T}$ with multiplicity $d$.

The definition of the spectral estimator $\hat Z_1,\ldots, \hat Z_n$ is given in (\ref{eqn:intro_hat_Z_def}).
The mapping $\mathp:\mathr^{d\times d}\rightarrow \matho(d)$ is from the polar decomposition and is defined as follows. For any matrix $B\in\mathr^{d\times d}$ that is full-rank, it  admits a singular value decomposition (SVD):  $B=MDV^\T$ with $M,V\in\matho(d)$ and $D$ a diagonal matrix. Then its polar decomposition is $B=(MV^\T)(VDV^\T)$ and $\mathp(B):= MV^\T$  is defined as its first factor.

Recall that $\check u$ is the leading eigenvector of $A$ and the population eigenspace $U^*$ is defined in (\ref{eqn:U_star_def}).
That is, $U^*\in\mathr^{nd\times d}$ and its $j$th submatrix is $U^*_j = \check u_j Z^*_j\in\mathr^{d\times d}$ for each $j\in[n]$. Following the proof of Lemma \ref{lem:no_additive_noise}, we can show $U^*$ is the leading eigenspace of $ (A\otimes J_d)\circ Z^*Z^{*\T}$:
\begin{lemma}\label{lem:no_additive_noise_od}
Denote $\lambda^*_1\geq \lambda_2^*\geq \ldots \geq \lambda^*_{nd}$ as the eigenvalues of $ (A\otimes J_d)\circ Z^*Z^{*\T}$. Then $\lambda^*_1= \lambda^*_2=\ldots =\lambda^*_d$, all equal the leading eigenvalue of $A$.  In addition, $\lambda^*_{d+1}$ is equal to the second largest eigenvalue of $A$. Furthermore, 
 $U^*$ is the eigenspace of $ (A\otimes J_d)\circ Z^*Z^{*\T}$ corresponding to $\lambda^*_1$, i.e., 
\begin{align*}
 ((A\otimes J_d)\circ Z^*Z^{*\T}) U^* = \lambda^*_1 U^*.
\end{align*}
\end{lemma}

Following the proof of  Proposition \ref{prop:no_additive_noise}, particularly using (\ref{eqn:no_additive_noise}),  we can further establish the exact recovery of $\hat Z$, up to an orthogonal matrix,  in the no-additive-noise case.

\begin{proposition}\label{prop:no_additive_noise_od}
Consider the no-additive-noise case where $\sigma=0$. There exists some constant $C_1>0$ such that if $\frac{np}{\log n}> C_1$, we have $\ell(\hat z,z^*)=0$ with probability at least $1-7n^{-10}$.
\end{proposition}

Similar to the phase synchronization, we can study the first-order approximation of the eigenspace $U$. Denote $\Lambda:=\text{diag}(\lambda_1,\ldots,\lambda_d)\in\mathr^{d\times d}$ as the diagonal matrix of the $d$ largest eigenvalues of $\mathx$. Then $U$ can be expressed as $U = \mathx U\Lambda^{-1}$. Define
\begin{align}\label{eqn:tilde_U_def}
\tilde U := \argmin_{U'\in\matho(nd,d)}\fnorm{U'-\mathx U^* }^2.
\end{align}
Then $\tilde U$ is the projection of $\mathx U^*$ onto $\matho(nd,d)$. This is similar to the definition of $\tilde u$ in (\ref{eqn:tilde_u_def_intro}) for the phase synchronization, where $\tilde u$ is the projection of $Xu^*$ onto the unit sphere. As a result, $\tilde U$ can be regarded as the first-order approximation of $U$.

The following lemma provides an upper bound for a leading eigenspace and its first-order approximation of two arbitrary Hermitian matrices. It is an extension of Lemma \ref{lem:eigenvector_perturbation} which is only about the perturbation of a leading eigenvector. The proof of Lemma \ref{lem:eigenspace_perturbation} follows that of Lemma \ref{lem:eigenvector_perturbation} but  is more involved, as  it needs to deal with matrix multiplication which is not commutative.

\begin{lemma}\label{lem:eigenspace_perturbation}
Consider two symmetric matrices $Y,Y^*\in\mathr^{n\times n}$. Let $\mu^*_1\geq\mu^*_2\geq \ldots \geq \mu^*_n$ be the  eigenvalues of  $Y^*$. Let $V^*\in\mathr^{n\times d}$ (resp. $V$) be the leading eigenspace of $Y^*$ (resp. $Y$) corresponding to its $d$ largest eigenvalues.
Define $\tilde V := \argmin_{V'\in\matho(n,d)}\fnorm{V'-YV^*}^2$.
If $\norm{Y -Y^*}\leq  \min\{\mu_d^*-\mu_{d+1}^*,\mu_d^*\}/4$, we have
\begin{align*}
&\inf_{O\in\matho(d)}\norm{V-\tilde V O} \leq  \frac{16\sqrt{2}}{3\br{\mu_d^* - \mu_{d+1}^*}\mu_d^*}\br{\frac{2\mu_1^*}{3(\mu_d^*-\mu_{d+1}^*)} + 1}\norm{Y-Y^*}^2 \\
&\quad + \frac{8\sqrt{2}}{3\br{\mu_d^* - \mu_{d+1}^*}\mu_d^*} \br{\frac{4\mu_1^*\br{\mu_1^*-\mu_d^*}}{\mu_d^* - \mu_{d+1}^*} + 2(\mu_1^* - \mu_d^*) +  \max\{|\mu^*_{d+1}|,|\mu^*_n|\}} \norm{Y-Y^*}.
\end{align*}
\end{lemma}

 Lemma \ref{lem:eigenspace_perturbation} includes Lemma \ref{lem:eigenvector_perturbation} as a special case when $d=1$. For $d>1$, if $\mu_1^* =\mu_d^*$, i.e., the largest $d$ eigenvalues of $Y^*$ are all equal,
 the upper bound in  Lemma \ref{lem:eigenspace_perturbation} simplifies to
\begin{align*}
\inf_{O\in\matho(d)}\|{V-\tilde V O}\| &\lesssim \frac{1}{\mu_d^* - \mu_{d+1}^*}\Bigg(\br{\frac{1}{\mu_d^* - \mu_{d+1}^*} + \frac{1}{\mu_d^*}} \norm{Y-Y^*}^2 \\
&\quad\quad\quad\quad\quad\quad + \frac{\max\{|\mu^*_{d+1}|,|\mu^*_n|\}}{\mu_d^*}\norm{Y-Y^*}\Bigg),
\end{align*} 
  which is similar in form to the upper bound in Lemma \ref{lem:eigenvector_perturbation}. This expression can be used in the $\matho(d)$ synchronization problem as $\lambda_1^*$ is shown to be equal to $\lambda_d^*$  in Lemma \ref{lem:no_additive_noise_od}.
A direct application of this expression leads to the  following proposition regarding the perturbation between $U$ and $\tilde U$.

\begin{proposition}\label{prop:eigenspace_perturbation}
Assume $2\leq d \leq C_0$ for some constant $C_0>0$.
There exist constants $C_1,C_2,C_3>0$ such that if $\frac{np}{\log n}>C_1$ and $\frac{np}{\sigma^2}>C_2$, we have
\begin{align*}
\inf_{O\in\matho(d)}\norm{U-\tilde UO}\leq C_3\frac{\sigma^2d + \sigma\sqrt{d}}{np},
\end{align*}
with probability at least $1-6n^{-10}$.
\end{proposition}

When $d=1$, Proposition \ref{prop:eigenspace_perturbation} reduces to Proposition \ref{prop:eigenvector_perturbation}. With  Proposition \ref{prop:eigenspace_perturbation},
we can  carry out a sharp $\ell_2$ analysis of the performance of the spectral estimator $\hat Z$ using $\tilde U$.
The loss function is defined analogously to (\ref{eqn:ell_def}) as
\begin{align*}
\ellod(\hat Z,Z^*):= \min_{O\in\matho(d)} \frac{1}{n}\fnorm{\hat Z_j -Z_j^*O}^2.
\end{align*}
In this way, we have the following theorem which is similar to Theorem \ref{thm:phase_main}. Its asymptotic version is given in Theorem \ref{thm:od_asymptotic}. The proof of Theorem \ref{thm:od} follows that of Theorem \ref{thm:phase_main} but is more complicated due to the existence of the mapping $\mathp$ in the definition of the spectral method. To prove Theorem \ref{thm:od}, note that for each $j\in[n]$, $\normf{\hat Z_j - Z^*_j} =\normf{\mathp(U_j)- Z^*_j} = \normf{\mathp(Z_j^{*\T}U_j)- I_d} $ where $Z_j^{*\T}U_j$ can be approximated by $Z_j^{*\T}\tilde U_j$ according to Proposition (\ref{prop:eigenspace_perturbation}). The term $Z_j^{*\T}\tilde U_j$ can be further expanded using (\ref{eqn:mathX_matrix_form}) and Lemma \ref{lem:no_additive_noise_od}, leading to $\sum_{k\neq j}A_{jk}Z_j^{*\T}\mathw_{jk}Z^*_k$ and several approximation error terms. Careful analysis of $\sum_{k\neq j}A_{jk}Z_j^{*\T}\mathw_{jk}Z^*_k$ eventually leads to the minimax risk $d(d-1)\sigma^2/(2np)$ and all the other error terms turn out to be negligible.

\begin{theorem}\label{thm:od}
Assume $2\leq d \leq C_0$ for some constant $C_0>0$.
There exist constants $C_1,C_2,C_3$  such that if $\frac{np}{\log n}>C_1$ and $\frac{np}{\sigma^2}>C_2$, we have
\begin{align*}
\ellod(\hat Z,Z^*)&\leq \br{1+ C_3\br{ \br{\frac{\sigma^2}{np}}^\frac{1}{4} + \sqrt{\frac{\log n}{np}} + \frac{1}{\log (np)}} }\frac{d(d-1)\sigma^2}{2np}
\end{align*}
holds with probability at least $1-n^{-9}-\ebr{-\frac{1}{32}\br{\frac{np}{\sigma^2}}^\frac{1}{4}}$.
\end{theorem}

We can compare the upper bound in Theorem \ref{thm:od}  to existing results for the $\matho(d)$ synchronization. \cite{gao2021optimal} derived an upper bound for the spectral method: $\ell(\hat Z,Z^*) \lesssim d^4(1+\sigma^2d)/(np)$ with high probability. In comparison, our upper bound has a smaller factor of $d(d-1)/2$  for $\sigma^2/np$. In addition, it does not have the $d^4/np$ error term. The paper \cite{gao2021optimal} also established the minimax lower bound: when $2\leq d\leq C_0$, there exist constants $C_4,C_5>0$ such that if $\frac{np}{\sigma^2}>C_4$, we have
\begin{align*}
\inf_{Z\in\matho(d)^n}\sup_{Z^*\in\matho(d)^n}\ellod(\hat Z,Z^*)\geq \br{1-C_5\br{\frac{1}{n} + \frac{\sigma^2}{np}}} \frac{d(d-1)\sigma^2}{2np}.
\end{align*}
Compared to the lower bound, the spectral estimator $\hat Z$ is exact minimax optimal as it achieves the correct rate with the correct constant $d(d-1)/2$ in front of the optimal rate $\sigma^2/np$.

\section{Discussions}\label{sec:discussions}

\subsection{Comparison of Spectral Method and Other Methods}\label{sec:comparison}

In synchronization problems, the spectral method offers computational advantages over alternative methods such as MLE, SDP, and GPM. According to Theorem \ref{thm:asymptotic}, the spectral method attains statistical optimality in the limit as \(\frac{np}{\sigma^2}\rightarrow\infty\), achieving the minimum possible risk. The performance of the spectral method in scenarios where \(\frac{np}{\sigma^2}\) does not approach infinity, however, remains less understood.

Previous studies \cite{javanmard2016phase, lelarge2019fundamental} have explored the PCA method in Bayesian settings for synchronization problems with \(p=1\). Unlike the spectral method, as defined in (\ref{eqn:spectral_estimator_def}), PCA does not involve entrywise normalization but scales the leading eigenvector \(u\) to minimize the mean square error (MSE). These studies offer a comprehensive asymptotic analysis of PCA's MSE and that of the Bayes-optimal estimator, demonstrating both methods' ability to achieve substantial accuracy when \(\sigma^2\) is below a specific threshold. However, PCA tends to exhibit a higher MSE compared to the Bayes-optimal estimator. Furthermore, \cite{javanmard2016phase} indicates that the MSE of SDP falls between that of PCA and the Bayes-optimal estimator, leaning more towards the latter.

While Theorem \ref{thm:asymptotic} addresses the regime where \(\frac{np}{\sigma^2}\rightarrow\infty\), Theorem \ref{thm:phase_main} establishes an upper bound in scenarios where \(\frac{np}{\sigma^2}\) exceeds a certain constant. This suggests a complex interplay between the performance of the spectral method and the ratio \(\frac{\sigma^2}{np}\) in the constant \(\frac{np}{\sigma^2}\) regime. To better understand this relationship, we conducted numerical experiments using the spectral method, GPM, and SDP under various \(\sigma^2\) levels. The GPM, initialized with the spectral estimator \(\hat z^{(0)}_\text{GPM}:=\hat z\), iteratively updates \(\hat z^{(t)}_\text{GPM}:=f(X \hat z^{(t-1)}_\text{GPM} )\) for \(t\geq 1\), where $f:\mathc^n\rightarrow \mathc_1^n$ is an entrywise normalization function defined as $[f(x)]_i := x_i/|x_i| \indic{x_i\neq 0} +\indic{x_i=0}$ for any $x\in\mathc^n$. The SDP, a convex optimization problem, maximizes $\max_{Z\in\mathc^{n\times n}:Z=Z^\H, \text{diag}(Z)=I_n,Z\succeq 0} \text{Tr}(XZ)$ over complex positive-semidefinite Hermitian matrices with unit diagonal entries and can be initialized using the spectral method. We assessed their performances using the normalized squared \(\ell_2\) loss (\ref{eqn:ell_def}).

\begin{figure}[ht]
\centering
\includegraphics[width=0.49\textwidth, trim= 0 15 25 55, clip]{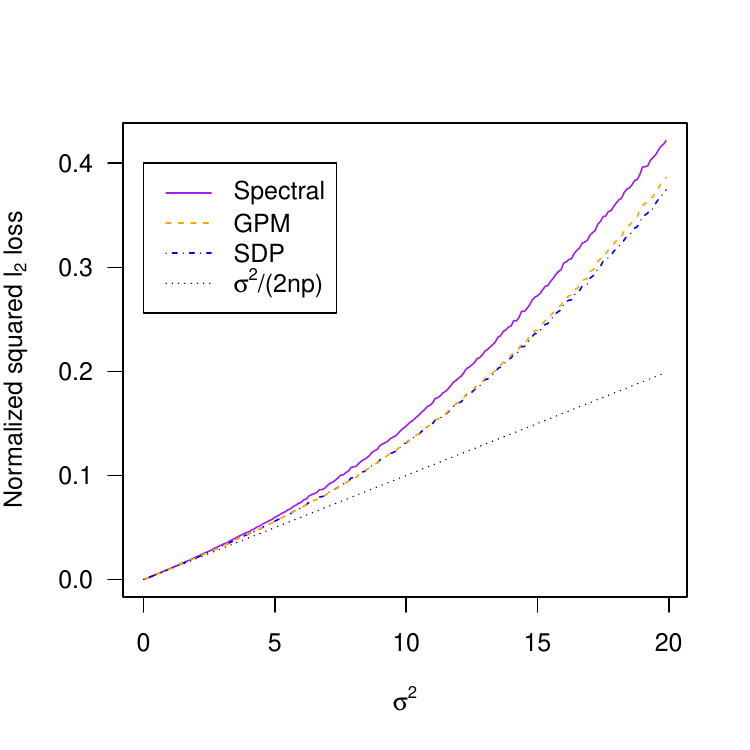}
\includegraphics[width=0.49\textwidth, trim= 0 15 25 55, clip]{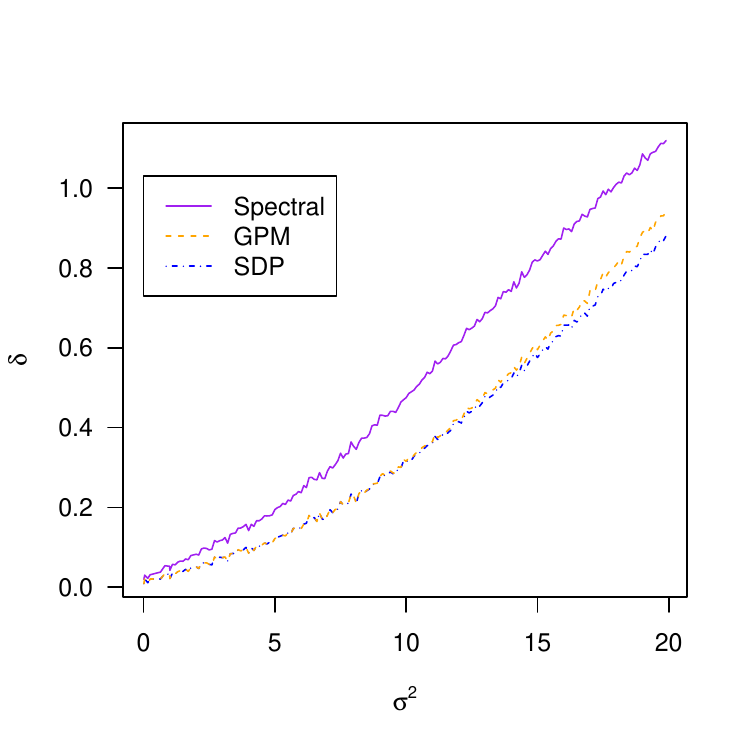}
\caption{\label{fig:1} Numerical results for the spectral method, GPM, and SDP in phase synchronization, with \(n=100\), \(p=0.5\) and \(\sigma^2\) varying within \([0,20]\). Left: Error comparison measured by the normalized squared \(\ell_2\) loss. Right: Comparison of the high-order term in their errors.}
\end{figure}

Figure \ref{fig:1} summarizes the comparative performances of these methods. For low \(\sigma^2\) values, the error rates of all methods approximate \(\frac{\sigma^2}{2np}\). The left panel of the figure shows that as \(\sigma^2\) increases, their error rates rise more steeply than \(\frac{\sigma^2}{2np}\). As \(\sigma^2\) continues to increase, the spectral method exhibits higher error rates, as expected, since the other two methods use the spectral method for initialization and enhance it through more complex procedures. For a deeper insight into the numerical performance differences, we compare the high-order terms in their errors. Specifically, the normalized squared \(\ell_2\) loss for each method can be expressed as \((1+\delta) \frac{\sigma^2}{2np}\), where \(\delta\) represents the high-order term. The right panel of Figure \ref{fig:1} compares \(\delta\) for these three methods. It reveals that even at small \(\sigma^2\) values, the spectral method's performance diverges from those of the other methods. This suggests that while \(\delta\) diminishes to 0 for all three methods as \(\sigma^2\) decreases (thus achieving exact minimax optimality), the spectral method's \(\delta\) diminishes more slowly than those of the other two methods.

Deriving explicit expressions for these error rates would be insightful, yet it falls outside the scope of this paper and presents an avenue for future research.

\subsection{Condition on $p$}\label{sec:condition_p}
In the phase synchronization problem (\ref{eqn:model}), observations are missing at random, forming an Erd\"{o}s-R\'enyi random graph $A$ with edge probability $p$. The value of $p$ cannot be excessively small, as this could result in $A$ being disconnected, thereby making accurate estimation of $z^*$ under a global phase impossible.
Theorem \ref{thm:asymptotic} assumes $\frac{np}{\log n}\rightarrow\infty$ to establish the exact minimax optimality of the spectral method. A less stringent condition, where $\frac{np}{\log n}$ exceeds a certain constant, is considered in Theorem \ref{thm:phase_main}. However, it is known that $A$ is connected with high probability when $\frac{np}{\log n}>1+\epsilon$ for any constant $\epsilon>0$. This raises the question of how the spectral method performs when $\frac{np}{\log n}$ is a small constant.

Our analysis requires  $\frac{np}{\log n}$ to be greater than a certain constant for several technical reasons. This condition ensures  desired bounds hold for critical quantities such as $\|A-\E A\|$ and $\|A\circ W\|$, which are essential for the $\ell_\infty$ analysis in Proposition \ref{prop:no_additive_noise} and the $\ell_2$ analysis of the first order approximation in Proposition \ref{prop:eigenvector_perturbation}. Moreover, the proof of Theorem \ref{thm:phase_main} leverages the $\ell_\infty$ results from Proposition \ref{prop:no_additive_noise}, leading to the inclusion of the $\sqrt{\frac{\log n}{np}}$ factor in the theorem's upper bound. This requires  $\frac{np}{\log n}$ to approach infinity for the upper bound to asymptotically match the exact minimax risk.
Obtaining precise bounds for the performance of the spectral method when $\frac{np}{\log n}$ is a small constant would require an extension beyond our current analytic framework, a task we leave for future research.

\subsection{Other Low-rank Problems}\label{sec:other_low_rank}

The synchronization problems investigated in this manuscript are part of a broader category of problems characterized by low-rank matrix structures disrupted by additive noise and incomplete data. The methodologies developed herein are applicable to a variety of related problems, such as matrix completion, principal component analysis, factor models,  mixture models, and ranking from pairwise comparison data.
A key observation is that many of these problems encompass multiple sources of randomness, such as that arising from missing data and additive noise. An effective approach, as demonstrated in this study, is to isolate these sources and evaluate their individual contributions to the overall estimation error. This strategy is exemplified in our analysis of synchronization problems, where we introduce a novel population eigenvector and eigenspace.
Furthermore, Lemma \ref{lem:eigenvector_perturbation} and Lemma \ref{lem:eigenspace_perturbation} offer a  general framework for the perturbation analysis of eigenvectors and eigenspaces. 

On the other hand, synchronization problems are special in that their leading eigenvector or eigenspace is spread out. In the literature \cite{chen2021spectral}, the coherence of an eigenvector \( u \)  is defined as \( \max_{i\in[n]}|u_i|^2/n \), where \( u_1, \ldots, u_n \) are its coordinates. In phase synchronization, the leading eigenvector of \( \E X \) in  (\ref{eqn:EX}) possesses uniformly equal magnitude \( 1/\sqrt{n} \), indicating maximal coherence. Contrastingly, in many low-rank problems, eigenvectors exhibit lower coherence, which naturally factors into theoretical analysis. Therefore, when extending the concepts and methodologies from this paper to other scenarios, it is crucial to monitor eigenvector coherence for more precise and insightful analysis.

\section{Proofs for Phase Synchronization}\label{sec:proof_phase} 

\subsection{Proof of Lemma \ref{lem:eigenvector_perturbation}}

We first present a variant of Davis-Kahan Theorem \cite{davis1970rotation} and an inequality about $\inf_{b\in\mathc_1}\norm{x - y b}$ and $\norm{(I_d-xx^\H)y}$ that will be used in the proof of Lemma \ref{lem:eigenvector_perturbation}.
\begin{lemma}\label{lem:davis}
Let $X,\tilde X\in\mathc^{d\times d}$ be two Hermitian matrices. Let $\lambda_1\geq \lambda_2\geq \ldots \geq \lambda_d$  be the eigenvalues of $X$. Consider any $r\in[d]$. Let $U\in\mathc^{d\times r}$ (resp. $\tilde U$) be the eigenspace of $X$ (resp. $\tilde X$) that includes its  leading $r$ eigenvectors.  Under the assumption that $\|{X-\tilde X}\|< (\lambda_r-\lambda_{r+1})/4$, we  have
\begin{align*}%
\norm{(I-UU^\H)\tilde U}
 \leq \frac{4\norm{X-\tilde X}}{3(\lambda_r - \lambda_{r+1})}.
\end{align*}
\end{lemma}

\begin{lemma}\label{lem:complex_unit_vector_inequality}
For any unit vectors $x,y\in\mathbb{C}^d$,
we have $\inf_{b\in\mathc_1}\norm{x - y b}\leq \sqrt{2}\norm{(I_d-xx^\H)y}$.
\end{lemma}

\begin{proof}[Proof of Lemma \ref{lem:eigenvector_perturbation}]
Denote $\mu_1\geq \ldots\geq \mu_n$ as the eigenvalues of $Y$. We first give some inequalities for the eigenvalues and $\|Yv^*\|$ that will be used later in the proof. By Weyl's inequality, we have
\begin{align*}
\max\cbr{\abs{\mu_1-\mu_1^*},\abs{\mu_2-\mu_2^*}}\leq \norm{Y-Y^*}.
\end{align*}
Since $\norm{Y-Y^*}\leq \min\{\mu_1^*-\mu_2^*,\mu_1^*\}/4$ is assumed, we have
\begin{align}\label{eqn:rev_4}
\frac{3}{4}\mu_1^*\leq \mu_1\leq \frac{5}{4}\mu_1^*,\quad\mu_1-\mu_2 \geq \frac{\mu_1^* -\mu_2^*}{2},
\end{align}
and
\begin{align}\label{eqn:rev_5}
\abs{\frac{\mu_1^*}{\mu_1} - 1}  = \frac{\abs{\mu_1^* -\mu_1}}{\mu_1} \leq \frac{\norm{Y-Y^*}}{\mu_1^* - \norm{Y-Y^*}} \leq  \frac{4\norm{Y-Y^*}}{3\mu_1^*}.
\end{align}
Regarding $\|Yv^*\|$, using   the decomposition
\begin{align*}
Y=Y^* + (Y-Y^*) = \mu_1^*v^*v^{*\H} + (Y^* - \mu_1^*v^*v^{*\H}) + (Y-Y^*),
\end{align*}
and its consequence
\begin{align}\label{eqn:proof_4}
Yv^* = Y^* v^* + (Y-Y^*)v^* =\mu_1^*v^* + (Y-Y^*)v^*,
\end{align}
we have 
\begin{align}\label{eqn:rev_2}
\norm{Yv^*}\geq \mu_1^* - \norm{Y-Y^*}\geq \frac{3\mu_1^*}{4}.
\end{align}

We define $\check v\in\mathc^n$ and $\tilde v\in\maths_n$ as
\begin{align}\label{eqn:check_v_def}
\check v&:= \frac{Y v^*}{\mu_1},\\
\tilde v  &:=\frac{Yv^*}{\norm{Yv^*}}.\label{eqn:tilde_v_def}
\end{align}
Then $\tilde v$ is the first-order approximation of $v$, written equivalently as $\tilde v={\check v}/{\norm{\check v}}$. Note that with $\norm{Yv^*}>0$ as shown in (\ref{eqn:rev_2}),  $\tilde v$ is well-defined.

Since $v$ is the eigenvector of $Y$ corresponding to $\mu_1$, we have
\begin{align*}
\mu_1 v &=Yv,\\
\mu_1 \tilde v &= Yv^*/\norm{\check v}.
\end{align*}
Subtracting the second equation from the first one,  we have
\begin{align*}
\mu_1 (v - \tilde v) = Y\br{v - \frac{v^*}{\norm{\check v}}} = Y(v-\tilde v) + Y\br{\tilde v - \frac{v^*}{\norm{\check v}}} =  Y(v-\tilde v)  + \frac{1}{\norm{\check v}} Y(\check v - v^*).
\end{align*}
After rearranging, we have
\begin{align}\label{eqn:proof_1}
\norm{\check v}(\mu_1 I_n - Y)(v - \tilde v ) =  Y(\check v - v^*).
\end{align}
Since $(\mu_1 I_n - Y)v=0$, we have $\text{span}(\mu_1 I_n-Y)$ being orthogonal to $v$.  As a result,  $\norm{\check v}(\mu_1 I_n- Y)(v - \tilde v ) = \norm{\check v}(\mu_1 I_n- Y)(I_n - vv^\H) \tilde v $.
In addition, since the left-hand side of (\ref{eqn:proof_1}) belongs to $\text{span}(I_n-vv^\H)$, its right-hand side must also  belong to  $\text{span}(I_n-vv^\H)$. That is, $Y(\check v - v^*) =(I_n-vv^\H) Y(\check v - v^*)$.
Then (\ref{eqn:proof_1}) leads to
\begin{align}\label{eqn:proof_2}
 \norm{\check v}(\mu_1 I_n- Y)(I_n - vv^\H) \tilde v = (I_n-vv^\H) Y(\check v - v^*).
\end{align}

Observe that  $0\leq \mu_1-\mu_2\leq\ldots \leq \mu_1-\mu_n$ are the eigenvalues of $\mu_1 I_n-Y$. In particular, the eigenvector corresponding to 0 is $v$.  Since $(I_n-vv^\H)\tilde v$ is orthogonal to $v$, from (\ref{eqn:proof_2}) we have
\begin{align*}
  \norm{\check v} (\mu_1-\mu_2)\norm{(I_n-vv^\H)\tilde v} \leq \norm{\check v}  \norm{(\mu_1 I_n- Y)(I_n-vv^\H)\tilde v} = \norm{ (I_n-vv^\H)Y(\check v - v^*)} .
\end{align*}
Hence,
\begin{align}\label{eqn:proof_3}
\norm{(I_n-vv^\H)\tilde v} \leq  \frac{1}{ \norm{\check v} (\mu_1-\mu_2)} \norm{ (I_n-vv^\H)Y(\check v - v^*)} .
\end{align}
From Lemma \ref{lem:complex_unit_vector_inequality}, we have $\inf_{b\in\mathc_1}\norm{v - \tilde v b}\leq \sqrt{2}\norm{(I_n-vv^\H)\tilde v}$.
With this, (\ref{eqn:proof_3}) leads to
\begin{align}\label{eqn:proof_3_1}
\inf_{b\in\mathc_1}\norm{v - \tilde v b} \leq  \frac{\sqrt{2}}{ \norm{\check v} (\mu_1-\mu_2)} \norm{ (I_n-vv^\H)Y(\check v - v^*)} .
\end{align}

In the following, we are going to analyze ${ (I_n-vv^\H)Y(\check v - v^*)}$.
We have
\begin{align*}
&(I_n-vv^\H)Y(\check v - v^*) \\
&=(I_n-vv^\H)Y \br{\frac{Yv^*}{\mu_1} - v^*} \\
&=(I_n-vv^\H)Y  \br{\frac{\mu^*_1}{\mu_1}-1}v^* +  \frac{1}{\mu_1} (I_n-vv^\H)Y\br{Y-Y^*}v^*\\
&= \br{\frac{\mu^*_1}{\mu_1}-1}(I_n-vv^\H)\mu^*_1v^* +  \br{\frac{\mu^*_1}{\mu_1}-1}(I_n-vv^\H)(Y-Y^*)v^* \\
&\quad + \frac{1}{\mu_1} (I_n-vv^\H) \mu_1^* v^*v^{*\H}\br{Y-Y^*}v^* + \frac{1}{\mu_1} (I_n-vv^\H) (Y^*-\mu_1^* v^*v^{*\H})\br{Y-Y^*}v^*  \\
&\quad +  \frac{1}{\mu_1} (I_n-vv^\H) (Y - Y^*)\br{Y-Y^*}v^* \\
&=  \br{ \br{\frac{\mu^*_1}{\mu_1}-1} + \frac{1}{\mu_1}v^{*\H}\br{Y-Y^*}v^*  }\mu^*_1(I_n-vv^\H)v^*  \\
&\quad +   \br{\frac{\mu^*_1}{\mu_1}-1}(I_n-vv^\H)(Y-Y^*)v^* +\frac{1}{\mu_1} (I_n-vv^\H) (Y^* - \mu_1^*v^*v^{*\H})\br{Y-Y^*}v^* \\
&\quad+  \frac{1}{\mu_1} (I_n-vv^\H) (Y - Y^*)\br{Y-Y^*}v^*.
\end{align*}
Hence,
\begin{align*}
&\norm{(I_n-vv^\H)Y(\check v - v^*) }\\
&\leq \br{\abs{\frac{\mu^*_1}{\mu_1}-1} + \frac{\abs{v^{*\H}\br{Y-Y^*}v^*}}{\mu_1}}\mu_1^*\norm{(I_n-vv^\H)v^*} + \abs{\frac{\mu^*_1}{\mu_1}-1} \norm{Y-Y^*}\\
&\quad + \frac{\norm{Y^* - \mu_1^*v^*v^{*\H}}\norm{Y-Y^*}}{\mu_1} + \frac{\norm{Y-Y^*}^2}{\mu_1}\\
&\leq \br{\abs{\frac{\mu^*_1}{\mu_1}-1} + \frac{\norm{Y-Y^*}}{\mu_1}}\mu_1^*\norm{(I_n-vv^\H)v^*} + \abs{\frac{\mu^*_1}{\mu_1}-1} \norm{Y-Y^*}\\
&\quad + \frac{|\mu_2^*|\norm{Y-Y^*}}{\mu_1} + \frac{\norm{Y-Y^*}^2}{\mu_1},
\end{align*}
where we use the fact that $\norm{I_n-vv^\H}=1$ and $\norm{Y^* - \mu_1^*v^*v^{*\H}} =\max\{|\mu_2^*|,|\mu_n^*|\}$. Then together with (\ref{eqn:proof_3_1}), we have
\begin{align*}
\inf_{b\in\mathc_1}\norm{v - \tilde v b} &\leq  \frac{\sqrt{2}}{ \norm{\check v} (\mu_1-\mu_2)}  \Bigg(\br{\abs{\frac{\mu^*_1}{\mu_1}-1} + \frac{\norm{Y-Y^*}}{\mu_1}}\mu_1^*\norm{(I_n-vv^\H)v^*} \\
&\quad + \abs{\frac{\mu^*_1}{\mu_1}-1} \norm{Y-Y^*}+ \frac{\max\{|\mu_2^*|,|\mu_n^*|\}\norm{Y-Y^*}}{\mu_1} + \frac{\norm{Y-Y^*}^2}{\mu_1}\Bigg).
\end{align*}

In the rest of the proof, we are going to simplify the display above. 
From (\ref{eqn:rev_4}) and  (\ref{eqn:rev_2}), we have
\begin{align*}
\norm{\check v} = \frac{\norm{Yv^*}}{\mu_1}\geq \frac{3}{5}. 
\end{align*}
Using Lemma \ref{lem:davis} and the assumption $\norm{Y - Y^*} \leq (\mu_1^* -\mu_2^*)/4$,
we have
\begin{align*}
\norm{(I_n-vv^\H)v^*} \leq 
\frac{2\norm{Y - Y^*} }{\mu_1^* -\mu_2^*}.
\end{align*}
With the above results, together with (\ref{eqn:rev_4}) and (\ref{eqn:rev_5}),  we have
\begin{align*}
&\inf_{b\in\mathc_1}\norm{v - \tilde v b}\\
 &\leq \frac{\sqrt{2}}{\frac{3}{5} \frac{\mu_1^*-\mu_2^*}{2}} \Bigg(\br{\frac{4\norm{Y-Y^*}}{3\mu_1^*} + \frac{\norm{Y-Y^*}}{\frac{3}{4}\mu_1^*}} \mu_1^* \frac{2\norm{Y - Y^*} }{\mu_1^*-\mu_2^*} + \frac{4\norm{Y-Y^*}}{3\mu_1^*}\norm{Y - Y^*}\\
&\quad + \frac{\max\{|\mu_2^*|,|\mu_n^*|\}\norm{Y-Y^*}}{\frac{3}{4}\mu_1^*}  + \frac{\norm{Y-Y^*}^2}{\frac{3}{4}\mu_1^*}\Bigg)\\
&=\frac{10\sqrt{2}}{3(\mu_1^*-\mu_2^*)}\br{\br{\frac{16}{3(\mu_1^{*}-\mu_2^*)}  + \frac{8}{3\mu_1^{*}} } \norm{Y-Y^*}^2+ \frac{4\max\{|\mu_2^*|,|\mu_n^*|\} }{3\mu_1^*}\norm{Y-Y^*}}\\
&\leq \frac{40\sqrt{2}}{9(\mu_1^*-\mu_2^*)}\br{\br{\frac{4}{(\mu_1^{*}-\mu_2^*)}  + \frac{2}{\mu_1^{*}} } \norm{Y-Y^*}^2+ \frac{\max\{|\mu_2^*|,|\mu_n^*|\} }{\mu_1^*}\norm{Y-Y^*}}.
\end{align*}
\end{proof}

\subsection{Proofs of Lemma \ref{lem:no_additive_noise}, Proposition \ref{prop:no_additive_noise}, and Proposition \ref{prop:eigenvector_perturbation}}

\begin{proof}[Proof of Lemma \ref{lem:no_additive_noise} ]
Denote $\lambda'$ as an eigenvalue of $A$ with its corresponding eigenvector $u'$. Then  we have $Au' =\lambda' u'$.  This can be equivalently written as
\begin{align*}
\sum_{k\neq j}A_{jk} u'_k = \lambda' u'_j,\forall j\in[n].
\end{align*}
Multiplying by ${z^*_j}$ on both sides, we have
\begin{align*}
\sum_{k\neq j}A_{jk} z^*_ju'_k  = \sum_{k\neq j}A_{jk} z^*_j \overline{z^*_k} (z^*_ku'_k) =  \lambda'z^*_j  u'_j,\forall j\in[n].
\end{align*}
That is, $(A\circ z^*z^{*\H})  (z^* \circ u') =\lambda'(z^* \circ u')$. Hence, $\lambda'$ is an eigenvalue of $A\circ z^*z^{*\H}$ with the corresponding eigenvector $z^* \circ u'$.

By the same argument, we can show each eigenvalue of $A\circ z^*z^{*\H}$ is also an eigenvalue of $A$. As a result, since $\check u$ is the leading eigenvector of $A$, $z^*\circ \check u$ is the leading eigenvector of $A\circ z^*z^{*\H}$.
\end{proof}

Before proving Proposition \ref{prop:no_additive_noise} and Proposition \ref{prop:eigenvector_perturbation}, we first state some technical lemmas related to $A$ and $W$.

\begin{lemma}\label{lem:A_related}
The largest eigenvalue of $\E A$ is $(n-1)p$ and the corresponding eigenvector is $\one_n/\sqrt{n}$. The remaining eigenvalues of $\E A$ are $-p$ with multiplicity $n-1$.  Denote $\lambda'\geq \lambda'_2\geq \ldots\geq \lambda'_n$ as the eigenvalues of $A$. 
We have
\begin{align}\label{eqn:rev_9}
&|\lambda' - (n-1)p|, \max_{2\leq j\leq n}|\lambda'_ j+p|\leq \norm{A-\E A}, \text{ and } \lambda'-\lambda'_2 \geq np-2\norm{A-\E A}.
\end{align}
\end{lemma}

\begin{lemma}\label{lem:combine}
There exist constants $C_1,C_2>0$ such that if $\frac{np}{\log n}>C_1$, then we have
\begin{align*}
&\norm{A-\E A}\leq C_2\sqrt{np},\\
&\norm{A\circ W}\leq C_2\sqrt{np},\\
&\sum_{j\in[n]}\abs{\im\br{{\sum_{k\neq j} A_{jk}W_{jk} \overline{ z_j^* } z^*_k }}}^2\leq \frac{n^2p}{2}\br{1+C_2\sqrt{\frac{\log n}{n}}},
\end{align*}
with probability at least $1-3n^{-10}$.
\end{lemma}

The first part of Proposition \ref{prop:no_additive_noise} (i.e., (\ref{eqn:no_additive_noise})) can be proved using Theorem 2.1 of \cite{abbe2020entrywise} which we include below for completeness. The statement of Theorem 2.1 in \cite{abbe2020entrywise} is complicated as the theorem works for perturbation of eigenspaces. However, what we need to consider here is only the perturbation of the leading eigenvector. For easier reference, we present below a simpler version of the theorem.

\begin{lemma}[A simpler version of Theorem 2.1 of \cite{abbe2020entrywise}]\label{lem:abbe}
Consider two symmetric matrices $Y,Y^*\in \mathr^{n\times n}$. Let  the eigenvalues of $Y^*$ be $\mu_1^* \geq \mu_2^*\geq \ldots \geq \mu_n^*$. Define $\Delta^*:=\min\{\mu^*_1 -\mu^*_2,\mu^*_1\}$ and $\kappa:=\max\{|\mu^*_1|,|\mu^*_n|\}/\Delta^*$. Let  the leading eigenvector of $Y$ (resp. $Y^*$) be $v$ (resp. $v^*$).
Assume the following conditions are satisfied for some $\gamma\geq 0$ and some function $\phi:[0,+\infty)\rightarrow[0,+\infty)$:
\begin{enumerate}
\item %
$\norm{Y^*}_{2\rightarrow\infty}\leq \gamma\Delta^*$.
\item For any $m\in[n]$, $\{Y_{jk}: j=m\text{ or }k=m\}$ are independent of $\{Y_{jk}: j\neq m,k\neq m\}$.
\item $32\kappa \max\{\gamma, \phi(\gamma)\}\leq 1$ and for some $\delta_0\in(0,1)$, $\pbr{\norm{Y-Y^*}\leq \gamma \Delta^*} \geq 1-\delta_0$.
\item Suppose $\phi(x)$ is continuous and non-decreasing in $[0,+\infty)$ with $\phi(0)=0$, $\phi(x)/x$ is non-increasing in $[0,+\infty)$, and $\delta_1\in(0,1)$. For any $m\in[n]$ and $w\in\mathr^n$,
\begin{align*}
\pbr{\abs{[Y-Y^*]_{m\cdot}w} \leq \Delta^* \norm{w}_\infty \phi\br{\frac{\norm{w}}{\sqrt{n} \norm{w}_\infty}}} \geq 1- \frac{\delta_1}{n}.
\end{align*}
\end{enumerate}
Then with probability at least $1-\delta_0-2\delta_1$, there exists some constant $C>0$ and some $b\in\{-1,1\}$ such that
\begin{align*}
\norm{vb-Yv^*/\mu_1^*}_\infty &\leq C\br{\kappa(\kappa + \phi(1))(\gamma + \phi(\gamma)) \norm{v^*}_\infty + \gamma \norm{Y^*}_{2\rightarrow\infty}/\Delta^*}. 
\end{align*}
\end{lemma}

The following Lemma \ref{lem:bern} provides two Bernstein-type concentration inequalities to be used in the proof of Proposition \ref{prop:no_additive_noise}. The first one is the classical Bernstein inequality; see Section 2.8 of \cite{boucheron2003concentration} for its proof. The second one is proved in Lemma 7 of \cite{abbe2020entrywise}.

\begin{lemma}\label{lem:bern}
Let $B_1,\ldots,B_n$ be real independent random variables such that $\max_{j\in[n]}\abs{B_j}\leq M$ for some $M>0$. Then
\begin{align*}
\pbr{\abs{\sum_{j\in[n]} (B_j - \E B_j)} \geq t  }\leq 2\ebr{- \frac{\frac{1}{2}t^2}{ \sum_{j\in[n]}\E(B_j - \E B_j)^2 + \frac{1}{3} Mt}}.
\end{align*}
Let $w\in\mathr^n$ be a fixed vector and $\alpha\geq 0$. If $\{B_j\}_{j\in[n]}\iid \text{Bernoulli}(p)$, we have
\begin{align*}
\pbr{\abs{\sum_{j\in[n]} w_j(B_j - p)}\geq \frac{(2+\alpha)np}{1\vee \log\br{\frac{\sqrt{n}\norm{w}_{\infty}}{\norm{w}}}}\norm{w}_\infty} \leq 2\exp(-\alpha np).
\end{align*}
\end{lemma}

\begin{proof}[Proof of Proposition \ref{prop:no_additive_noise}]
We use Lemma \ref{lem:abbe} to prove the first part of the proposition. Denote $\mu_1^* \geq \mu_2^*\geq \ldots \geq \mu_n^*$ as the eigenvalues of $\E A$. Define $\Delta^*$ and $\kappa$ the same as in Lemma \ref{lem:abbe}. From Lemma \ref{lem:A_related}, we have $\Delta^* = (n-1)p$, $\kappa =1$, with  $\one_n/\sqrt{n}$ being the leading eigenvector of $\E A$.  Since $\E A =pJ_n - pI_n$, we have $\|\E A\|_{2\rightarrow\infty} =\sqrt{(n-1)}p$.  By Lemma \ref{lem:combine}, there exist constants $c_1,c_2>1$ such that if $\frac{np}{\log n}>c_1$, then $\norm{A-\E A}\leq c_2\sqrt{np}$ with probability at least $1-3n^{-10}$. Define $\gamma:= 2c_2/\sqrt{np}$, $\delta_0:=2n^{-10}$, and $\phi(x) := 3(1\vee \log (x^{-1}))^{-1}$. Then the first assumption of Lemma \ref{lem:abbe} is satisfied as long as $c_2\geq 1$.
When $\frac{np}{\log n }$ is greater than some sufficiently large constant, we have $\phi(\gamma)\leq 8/\log (np)$, and the  third assumption  is satisfied. We can also verify that the second assumption is also satisfied. For any $m\in[n]$ and any $w\in\mathr^n$, since $[A-\E A]_{m\cdot} w$ is a weighted average of centered Bernoulli random variables, 
the second inequality of Lemma \ref{lem:bern} can be applied to have
\begin{align*}
&\pbr{\abs{[A-\E A]_{j\cdot} w}>\Delta^* \|w\|_{\infty} \phi\br{\frac{\norm{w}}{\sqrt{n} \norm{w}_\infty}}} \\
&\leq \pbr{\abs{[A-\E A]_{j\cdot} w}\geq \frac{2.5np}{1\vee \log\br{\frac{\sqrt{n}\norm{w}_{\infty}}{\norm{w}}}}\norm{w}_\infty} \leq 2n^{-11},
\end{align*}
when $\frac{np}{\log n}\geq 11$ is greater than some sufficiently large constant.
Define $\delta_1:=2n^{-10}$. Then the last assumption of Lemma \ref{lem:abbe} is  satisfied. Then Lemma \ref{lem:abbe} leads to the conclusion that with probability at least $1-6n^{-10}$, there exists some constant $c_1>0$ and some $b\in\{-1,1\}$ such that
\begin{align*}
\norm{\check u b- \frac{1}{\mu_1^*\sqrt{n}}A\one_n}_\infty &\leq c_1\br{\kappa(\kappa + \phi(1))(\gamma + \phi(\gamma)) \norm{\frac{1}{\sqrt{n}}\one_n}_\infty + \gamma \frac{\norm{\E A}_{2\rightarrow\infty}}{\Delta^*}} \\
&\leq c_1\br{(1+3)\br{\frac{2c_2}{\sqrt{np}} + \frac{8}{\log (np)}}\frac{1}{\sqrt{n}} + \frac{2c_2}{\sqrt{np}} \frac{\sqrt{(n-1)}p}{(n-1)p}} \\
&\leq \frac{c_2}{\log (np)} \frac{1}{\sqrt{n}},
\end{align*}
for some constant $c_2>0$. Note that
\begin{align*}
 \frac{1}{\mu_1^*\sqrt{n}}A\one_n =  \frac{1}{\mu_1^*\sqrt{n}}\E A\one_n +  \frac{1}{\mu_1^*\sqrt{n}}(A-\E A)\one_n = \frac{1}{\sqrt{n}} \one_n +  \frac{1}{(n-1)p\sqrt{n}} (A-\E A)\one_n.
\end{align*}
Then we have
\begin{align*}
\norm{\check u b- \frac{1}{\sqrt{n}} \one_n}_\infty \leq  \frac{c_2}{\log (np)} \frac{1}{\sqrt{n}} + \frac{1}{(n-1)p} \norm{\frac{1}{\sqrt{n}}(A-\E A)\one_n }_\infty.
\end{align*}
For any $m\in[n]$, by the first inequality of Lemma \ref{lem:bern}, there exists some constant $c_3>0$ such that 
\begin{align*}
\pbr{\abs{ [A-\E A]_{m\cdot}\one_n } \geq c_3 \sqrt{np \log n} } &\leq 2\ebr{-\frac{\frac{c_3^2}{2}np\log n}{(n-1)p(1-p) + \frac{c_3}{3} \sqrt{np\log n}}}\\
& \leq 2n^{-11}.
\end{align*}
Together with a union bound, we have $\pbr{\|(A-\E A)\one_n\| \geq c_3 \sqrt{np \log n} } \leq 2n^{-10}$. Hence,
\begin{align*}
\norm{\check u b- \frac{1}{\sqrt{n}} \one_n}_\infty \leq  \frac{c_2}{\log (np)} \frac{1}{\sqrt{n}} + \frac{1}{(n-1)p} \frac{c_3\sqrt{np \log n}}{\sqrt{n}} \leq c_4 \br{\sqrt{\frac{\log n}{np}} + \frac{1}{\log (np)}} \frac{1}{\sqrt{n}},
\end{align*}
for some constant $c_4>0$ with probability at least $1-8n^{-10}$.

The second part of the proposition is an immediate consequence of the first part. 
If $\frac{np}{\log n}> \max\cbr{C_1, 2C_2^2}$,
all the coordinates of $\check u$ have the same sign according to (\ref{eqn:no_additive_noise}). From Lemma \ref{lem:no_additive_noise}, we have $u =u^*$ as $u^*$ is the leading eigenvector of $A\circ z^*z^{*\H}$.
 If  $\{\check u_j\}_{j\in[n]}$ are all positive, we have $$\hat z_j = u_j^*/|{u_j^*}|=z^*_j \check u_j  / \check u_j  = z^*_j ,$$
 for each $j\in[n]$. That is, $\hat z=z^*$. If  $\{\check u_j\}_{j\in[n]}$ are all negative, we then have $\hat z=-z^*$.
\end{proof}

\begin{proof}[Proof of Proposition \ref{prop:eigenvector_perturbation}]
Recall $\lambda^*$ is the largest eigenvalue of $A\circ z^*z^{*\H}$. From Lemma \ref{lem:no_additive_noise}, $u^*$ is the corresponding eigenvector. 
Denote $\lambda^*_2\geq \ldots \geq \lambda^*_n$ as its remaining eigenvalues.  By Lemma \ref{lem:combine}, there exist constants $c_1,c_2>0$ such that when $\frac{np}{\log n}>c_1$, we have $\norm{A-\E A}\leq c_2\sqrt{np}$ and $\norm{A\circ W}\leq c_2\sqrt{np}$ with probability at least $1-3n^{-10}$. By Lemma \ref{lem:no_additive_noise} and Lemma \ref{lem:A_related}, we have $\lambda^* \geq (n-1)p - c_2\sqrt{np}$, $\max\{|\lambda^*_2|,|\lambda^*_n|\}\leq p + c_2\sqrt{np}$, and  $\lambda^* -\lambda_2^*\geq np - 2c_2\sqrt{np}$. When $\frac{np}{\log n}$ and $\frac{np}{\sigma^2}$ are greater than some sufficiently large constant, we have %
$4\sigma \norm{A\circ W} \leq np/2 \leq  \min\{\lambda^*_1,\lambda^* -\lambda_2^*\}$ satisfied. Since $X - A\circ z^*z^{*\H} = \sigma A \circ W$,  a direct application of Lemma \ref{lem:eigenvector_perturbation} leads to
\begin{align*}
&\inf_{b\in\mathc_1} \norm{u-\tilde ub}\\ 
&\leq \frac{40\sqrt{2}}{9(\lambda^* -\lambda_2^*)} \br{\br{\frac{4}{\lambda^* -\lambda_2^*}+\frac{2}{\lambda^{*}}}\sigma^2 \norm{A\circ W}^2 + \frac{\max\{|\lambda^*_2|,|\lambda^*_n|\}\sigma\norm{A\circ W}}{\lambda^*}}\\
&\leq   \frac{40\sqrt{2}}{9np/2} \br{\br{\frac{4}{np/2}+\frac{2}{np/2}}c_2^2\sigma^2 np + \frac{(p+c_2\sqrt{np})c_2\sigma\sqrt{np}}{np/2}}\\
&\leq c_3\frac{\sigma^2 + \sigma}{np},
\end{align*}
for some constant $c_3>0$.
\end{proof}

\subsection{Proof of  Theorem \ref{thm:phase_main}}

We first state some technical lemmas that will be used in the proof of Theorem \ref{thm:phase_main}.
\begin{lemma}\label{lem:re_exponential}
There exists some constant $C_1>0$ such that for any $\gamma$ satisfying $\frac{\gamma^2 np}{\sigma^2}\geq C_1$, we have
\begin{align*}
\sum_{j\in[n]}\indic{\frac{2\sigma}{np}  \abs{ {{\sum_{k\neq j} A_{jk}W_{jk} \overline{ z_j^* } z^*_k }} } \geq\gamma} \leq \frac{4\sigma^2}{\gamma^2 p}\ebr{-\frac{1}{16}\sqrt{\frac{\gamma^2 np}{\sigma^2}}},
\end{align*}
holds with probability at least $1-\ebr{-\frac{1}{32}\sqrt{\frac{\gamma^2 np}{\sigma^2}}}$.
\end{lemma}

\begin{lemma}[Lemma 10 and Lemma 11 of  \cite{gao2021exact}]\label{lem:x_normalize_y_diff}
For any $x\in\mathbb{C}$ such that $\re(x)>0$, $\left|\frac{x}{|x|}-1\right|\leq \left|\frac{\im(x)}{\re(x)}\right|$. For any $x\in\mathc\setminus\{0\}$ and any $y\in\mathc_1$, we have $\abs{\frac{x}{\abs{x}}-y}\leq 2\abs{x-y}$. 
\end{lemma}

\begin{proof}[Proof of Theorem \ref{thm:phase_main}]
Let $ b_1\in\mathc_1$ satisfy $\|{u - \tilde u  b_1}\|  = \inf_{a\in\mathc_1}\norm{u - \tilde u a} $. Denote $\delta:= u-\tilde u b_1 \in\mathc^n$. Recall $\check u$ is the leading eigenvector of $A$.
From Proposition \ref{prop:no_additive_noise}, Proposition \ref{prop:eigenvector_perturbation}, and Lemma \ref{lem:combine}, there exist constants $c_1,c_2>0$ such that if $\frac{np}{\log n},\frac{np}{\sigma^2}>c_1$, we have
\begin{align}
\norm{\delta} &\leq c_2 \frac{\sigma^2+\sigma}{np}, \label{eqn:new_proof_6}\\
\max_{j\in[n]}\abs{\check u_j - \frac{1}{\sqrt{n}}b_2}&\leq c_2\br{\sqrt{\frac{\log n}{np}} + \frac{1}{\log (np)}}\frac{1}{\sqrt{n}},\label{eqn:new_proof_7}\\
\norm{A-\E A} &\leq c_2\sqrt{np},\label{eqn:new_proof_8}\\
\norm{A\circ W} &\leq c_2\sqrt{np},\label{eqn:new_proof_9}\\
\sum_{j\in[n]}\abs{\im\br{{\sum_{k\neq j} A_{jk}W_{jk} \overline{ z_j^* } z^*_k }}}^2&\leq \frac{n^2p}{2}\br{1+c_2\sqrt{\frac{\log n}{n}}}, \label{eqn:new_proof_12}
\end{align}
with probability at least $1-n^{-9}$, for some $b_2\in\{-1,1\}$.

From (\ref{eqn:new_proof_7}), when $\frac{np}{\log n}\geq 2c_2^2$, $\check u$ is closer to $\one_n/\sqrt{n} b_2$ than to  $-\one_n/\sqrt{n} b_2$ with respect to $\ell_2$ norm. From Lemma \ref{lem:A_related}, $\one_n/\sqrt{n}$ is the leading eigenvector of $\E A$. By Lemma \ref{lem:davis} and Lemma \ref{lem:complex_unit_vector_inequality}, we have
\begin{align*}
\norm{\check u -\one_n/\sqrt{n} b_2} \leq \sqrt{2}\|(I -\one_n\one_n^\T/n ) \check u\| \leq \frac{2\norm{A-\E A}}{np} \leq \frac{2c_2}{\sqrt{np}}.
\end{align*}
Recall that $u^*$ is defined as $z^*\circ \check u$ in (\ref{eqn:u_star_def}). Define $\delta^* := u^* - \frac{1}{\sqrt{n}}z^*b_2$. 
This yields
\begin{align}
\norm{\delta^*} &= \norm{z^*\circ \check u - \frac{1}{\sqrt{n}}z^* \circ \one_n b_2} = \norm{z^* \circ \br{\check u- \frac{1}{\sqrt{n}}\one_n b_2}} \nonumber \\
& = \norm{\check u- \frac{1}{\sqrt{n}}\one_n b_2}  \leq \frac{2c_2\sqrt{np} + 2p}{np}.\label{eqn:new_proof_10}
\end{align}

By the definition of $\tilde u$ in (\ref{eqn:tilde_u_def_intro}), we can decompose $u$  into
\begin{align}
u &= \tilde u  b_1 + \delta  = \frac{Xu^*}{\norm{Xu^*}}b_1+\delta = \frac{b_1}{\norm{Xu^*}} \br{\br{A\circ z^*z^{*\H}}u^*  + \sigma \br{A\circ W }u^*} + \delta\nonumber\\
&=\frac{b_1}{\norm{Xu^*}}  \br{\lambda^* u^* + \sigma \br{A\circ W }u^*}+ \delta,\label{eqn:new_proof_1}
\end{align}
where we use the fact that $u^*$ is the eigenvector of $A\circ z^*z^{*\H}$ corresponding to the eigenvalue $\lambda^*$ by Lemma \ref{lem:no_additive_noise}.
With the definition of $u^*$ and also its approximation $\frac{1}{\sqrt{n}}z^*b_2$, (\ref{eqn:new_proof_1}) leads to
\begin{align*}
u =\frac{b_1}{\norm{Xu^*}}  \br{\lambda^*(z^* \circ \check u) + \sigma \br{A\circ W }\br{ \frac{1}{\sqrt{n}}z^*b_2 + \delta^*}}+ \delta.
\end{align*}
For any $j\in[n]$, denote $[A\circ W]_{j\cdot}$ as its $j$th row. From the  display above, we can express $u_j$ as
\begin{align*}
u_j = \frac{b_1}{\norm{Xu^*}} \br{\lambda^*  z^*_j \check u_j + \frac{\sigma}{\sqrt{n}}  \sum_{k\neq j} A_{jk}W_{jk} z^*_k b_2 + \sigma [A\circ W]_{j\cdot} \delta^*} + \delta_j.
\end{align*} 
By  (\ref{eqn:spectral_estimator_def}), when $u_j\neq 0$, we have
\begin{align}
\abs{\hat z_j - z_j^*b_1b_2 } &= \abs{b_2\overline{b_1 z_j^*} \hat z_j-1}= \abs{b_2\overline{b_1 z_j^*} \frac{u_j}{\abs{u_j}}-1} = \abs{\frac{b_2\overline{b_1 z_j^*} u_j}{\abs{b_2\overline{b_1 z_j^*} u_j}}-1}\nonumber\\
&=\abs{\frac{\frac{\norm{Xu^*}}{\lambda^*}b_2\overline{b_1 z_j^*} u_j}{\abs{\frac{\norm{Xu^*}}{\lambda^*}b_2\overline{b_1z_j^*} u_j}}-1}\label{eqn:rev_14}
\end{align}
which is all about $\frac{\norm{Xu^*}}{\lambda^*}b_2\overline{b_1 z_j^*} u_j$. 
With
\begin{align*}
\xi_j := \sum_{k\neq j} A_{jk}W_{jk}  \overline{z^*_j} z^*_k,
\end{align*}
we have
\begin{align}\label{eqn:new_proof_2}
\frac{\norm{Xu^*}}{\lambda^*}b_2\overline{b_1z_j^*} u_j &=  {b_2}\check u_j + \frac{\sigma}{\lambda^*\sqrt{n}}\xi_{j}  +  \frac{\sigma [A\circ W]_{j\cdot}\delta^* b_2\overline{ z_j^*}}{\lambda^*}+ \frac{\norm{Xu^*}}{\lambda^*} \delta_j b_2\overline{b_1 z_j^*}.
\end{align}
Note that  from (\ref{eqn:new_proof_7}), we have
\begin{align}\label{eqn:new_proof_3}
 b_2 \check u_j \geq  \br{1-c_2\br{\sqrt{\frac{\log n}{np}} + \frac{1}{\log (np)}}}\frac{1}{\sqrt{n}}.
\end{align}

Let $0<\gamma,\rho<1/8$ whose values will be given later. Consider the following two cases.

(1) If
\begin{align}
\abs{\frac{\sigma}{\lambda^*\sqrt{n}}\xi_{j}} &\leq \frac{\gamma}{\sqrt{n}}, \label{eqn:new_proof_4}\\
\abs{  \frac{\sigma [A\circ W]_{j\cdot}\delta^* }{\lambda^*}} &\leq \frac{\rho}{\sqrt{n}},\\
\abs{\frac{\norm{Xu^*}}{\lambda^*} \delta_j } &\leq \frac{\rho}{\sqrt{n}} \label{eqn:new_proof_5}
\end{align}
all hold, then from (\ref{eqn:new_proof_2}) and (\ref{eqn:new_proof_3}), we have 
\begin{align*}
\re\br{\frac{\norm{Xu^*}}{\lambda^*}b_2\overline{b_1z_j^*} u_j } \geq \br{1-c_2\br{\sqrt{\frac{\log n}{np}} + \frac{1}{\log (np)}} - \gamma - 2\rho}\frac{1}{\sqrt{n}},
\end{align*}
which can be further lower bounded by $1/(2\sqrt{n})$ for sufficiently large $\frac{np}{\log n}$. Therefore, $u_j\neq 0$ in this case.
Then by Lemma \ref{lem:x_normalize_y_diff} and (\ref{eqn:rev_14}), we have
\begin{align*}
&\abs{\hat z_j - z_j^*b_1b_2 } \\
 &\leq \frac{\abs{\im\br{\frac{\sigma}{\lambda^*\sqrt{n}}\xi_{j}  +  \frac{\sigma [A\circ W]_{j\cdot}\delta^* b_2\overline{ z_j^*}}{\lambda^*}+ \frac{\norm{Xu^*}}{\lambda^*} \delta_j b_2\overline{b_1 z_j^*}}}}{ \br{1-c_2\br{\sqrt{\frac{\log n}{np}} + \frac{1}{\log (np)}} - \gamma - 2\rho}\frac{1}{\sqrt{n}}} \\
&\leq \frac{\abs{\im\br{\frac{\sigma}{\lambda^*\sqrt{n}}\xi_{j} }}}{ \br{1-c_2\br{\sqrt{\frac{\log n}{np}} + \frac{1}{\log (np)}} - \gamma - 2\rho}\frac{1}{\sqrt{n}}}  + \frac{\abs{  \frac{\sigma [A\circ W]_{j\cdot}\delta^* }{\lambda^*}}  + \abs{\frac{\norm{Xu^*}}{\lambda^*} \delta_j } }{\frac{1}{2\sqrt{n}}} \\
& = \frac{\frac{\sigma}{\lambda^*}\abs{\im\br{\xi_{j} }}}{ \br{1-c_2\br{\sqrt{\frac{\log n}{np}} + \frac{1}{\log (np)}} - \gamma - 2\rho}} +  \frac{2\sqrt{n}\sigma}{\lambda^*}  \abs{[A\circ W]_{j\cdot}\delta^*} +\frac{2\sqrt{n}\norm{Xu^*}}{\lambda^*} \abs{\delta_j}.
\end{align*}
Note that for any $x,y\in\mathr$ and any $\eta>0$, we have $(x+y)^2 =x^2 + 2(\eta^{1/2} x)(\eta^{-1/2}y)+y^2\leq (1+\eta)x^2+(1+\eta^{-1})y^2$. We have
\begin{align*}
\abs{\hat z_j - z_j^*b_1b_2 }^2 &\leq  \frac{(1+\eta)\frac{\sigma^2}{\lambda^{*2}}\abs{\im\br{\xi_{j} }}^2}{ \br{1-c_2\br{\sqrt{\frac{\log n}{np}} + \frac{1}{\log (np)}} - \gamma - 2\rho}^2} \\
&\quad  +  (1+\eta^{-1}) \frac{8n\sigma^2}{\lambda^{*2}}\abs{[A\circ W]_{j\cdot}\delta^*} ^2+(1+\eta^{-1}) \frac{8n\norm{Xu^*}^2}{\lambda^{*2}}\abs{\delta_j}^2,
\end{align*}
where the value of $\eta>0$ will be given later.

(2) If any one of (\ref{eqn:new_proof_4})-(\ref{eqn:new_proof_5}) does not hold, we simply upper bound $|\hat z_j - z_j^*b_1b_2| $ by 2. Then this case can be written as
\begin{align*}
&\abs{\hat z_j - z_j^*b_1b_2 }^2\\
 &\leq 4\br{\indic{\abs{\frac{\sigma}{\lambda^*\sqrt{n}}\xi_{j}} > \frac{\gamma}{\sqrt{n}}} + \indic{\abs{  \frac{\sigma [A\circ W]_{j\cdot}\delta^* }{\lambda^*}} > \frac{\rho}{\sqrt{n}}} + \indic{\abs{\frac{\norm{Xu^*}}{\lambda^*} \delta_j } > \frac{\rho}{\sqrt{n}}}} \\
&\leq 4\br{ \indic{\sigma\abs{\xi_j} \geq {\gamma\lambda^*}}+ \frac{\sigma^2n\abs{[A\circ W]_{j\cdot}\delta^* }^2}{\rho^2\lambda^{*2}}+ \frac{n\norm{Xu^*}^2\abs{\delta_j}^2}{\rho^2\lambda^{*2}}},
\end{align*}
where in the last inequality we use the fact $\indic{x\geq y}\leq x^2/y^2$ for any $x,y>0$.

Combining the above two cases together, we have
\begin{align*}
&\abs{\hat z_j - z_j^*b_1b_2 }^2 \\
&\leq  \frac{(1+\eta)\frac{\sigma^2}{\lambda^{*2}}\abs{\im\br{\xi_{j}  }}^2}{ \br{1-c_2\br{\sqrt{\frac{\log n}{np}} + \frac{1}{\log (np)}} - \gamma - 2\rho}^2} \\
&\quad  +  (1+\eta^{-1}) \frac{8n\sigma^2}{\lambda^{*2}}\abs{[A\circ W]_{j\cdot}\delta^*} ^2+(1+\eta^{-1}) \frac{8n\norm{Xu^*}^2}{\lambda^{*2}}\abs{\delta_j}^2\\
&\quad +  4\br{ \indic{\sigma\abs{\xi_j} \geq {\gamma\lambda^*}}+ \frac{\sigma^2n\abs{[A\circ W]_{j\cdot}\delta^* }^2}{\rho^2\lambda^{*2}}+ \frac{n\norm{Xu^*}^2\abs{\delta_j}^2}{\rho^2\lambda^{*2}}}\\
&\leq   \frac{(1+\eta)\frac{\sigma^2}{\lambda^{*2}}\abs{\im\br{\xi_{j}  }}^2}{ \br{1-c_2\br{\sqrt{\frac{\log n}{np}} + \frac{1}{\log (np)}} - \gamma - 2\rho}^2} +  4\indic{\sigma\abs{\xi_j} \geq {\gamma\lambda^*}}\\
&\quad   + 8 (1+\eta^{-1}+\rho^{-2}) \frac{n\sigma^2}{\lambda^{*2}}\abs{[A\circ W]_{j\cdot}\delta^*} ^2 +8(1+\eta^{-1} + \rho^{-2}) \frac{n\norm{Xu^*}^2}{\lambda^{*2}}\abs{\delta_j}^2.
\end{align*}
The  display above holds for each $j\in[n]$. Summing over $j$, we have
\begin{align*}
&n\ell(\hat z, z^*) \\
&\leq \sum_{j\in[n]} \abs{\hat z_j - z_j^*b_1b_2 }^2\\
&\leq  \frac{(1+\eta)\frac{\sigma^2}{\lambda^{*2}}}{ \br{1-c_2\br{\sqrt{\frac{\log n}{np}} + \frac{1}{\log (np)}} - \gamma - 2\rho}^2} \sum_{j\in[n]}\abs{\im\br{\xi_{j}  }}^2 +  4 \sum_{j\in[n]}\indic{\sigma\abs{\xi_j} \geq {\gamma\lambda^*}}\\
&\quad   + 8 (1+\eta^{-1}+\rho^{-2}) \frac{n\sigma^2}{\lambda^{*2}} \sum_{j\in[n]}\abs{[A\circ W]_{j\cdot}\delta^*} ^2 +8(1+\eta^{-1} + \rho^{-2}) \frac{n\norm{Xu^*}^2}{\lambda^{*2}} \sum_{j\in[n]}\abs{\delta_j}^2\\
&\leq  \frac{(1+\eta)\frac{\sigma^2}{\lambda^{*2}}}{ \br{1-c_2\br{\sqrt{\frac{\log n}{np}} + \frac{1}{\log (np)}} - \gamma - 2\rho}^2} \sum_{j\in[n]}\abs{\im\br{\xi_{j}  }}^2 +  4 \sum_{j\in[n]}\indic{\sigma\abs{\xi_j} \geq {\gamma\lambda^*}}\\
&\quad   + 8 (1+\eta^{-1}+\rho^{-2}) \frac{n\sigma^2}{\lambda^{*2}} \norm{A\circ W}^2\norm{\delta^*}^2 +8(1+\eta^{-1} + \rho^{-2}) \frac{n\norm{Xu^*}^2}{\lambda^{*2}} \norm{\delta}^2,
\end{align*}
where in the last inequality, we use $\sum_{j\in[n]}\abs{[A\circ W]_{j\cdot}\delta^*} ^2 = \norm{(A\circ W)\delta^*}^2\leq  \norm{A\circ W}^2\norm{\delta^*}^2$.

We are going to simplify the  display above.  From (\ref{eqn:new_proof_6}), (\ref{eqn:new_proof_9}), (\ref{eqn:new_proof_10}), and  (\ref{eqn:new_proof_12}), we have upper bounds for $\norm{\delta},\norm{A\circ W},$  $\norm{\delta^*}$, and $\sum_{j\in[n]}\abs{\im\br{\xi_{j}}}^2 $. Using (\ref{eqn:new_proof_8}), Lemma \ref{lem:no_additive_noise}, and Lemma \ref{lem:A_related}, we have $\lambda^*\geq (n-1)p-c_2\sqrt{np}$
and a crude bound $np/2\leq \lambda^*\leq 2np$ when $\frac{np}{\log n}$ is greater than some sufficiently large constant. Due to the decomposition $X = A\circ z^*z^{*\H} + \sigma A\circ W$ and that $( A\circ z^*z^{*\H})u^* = \lambda^*u^*$, we have
\begin{align*}
\norm{Xu^*}  =\norm{\lambda^* u^* +  \sigma A\circ W u^* }\leq \lambda^* + \sigma\norm{ A\circ W} \leq np + c_2\sigma \sqrt{np}.
\end{align*}
From Lemma  \ref{lem:re_exponential}, if $\gamma$ satisfies $\frac{\gamma^2np}{\sigma^2}>c_3$ for some constant $c_3>0$, we have
\begin{align*}
\sum_{j\in[n]}\indic{\sigma\abs{\xi_j} \geq {\gamma\lambda^*}} &\leq \sum_{j\in[n]}\indic{\frac{2\sigma}{np}\abs{\xi_j} \geq {\gamma}}  \leq \frac{4\sigma^2}{\gamma^2 p}\ebr{-\frac{1}{16}\sqrt{\frac{\gamma^2 np}{\sigma^2}}},
\end{align*}
holds with probability at least $1-\ebr{-\frac{1}{32}\sqrt{\frac{\gamma^2 np}{\sigma^2}}}$. When $c_3$ is sufficiently large, we have 
\begin{align*}
\frac{4\sigma^2}{\gamma^2 np}\ebr{-\frac{1}{16}\sqrt{\frac{\gamma^2 np}{\sigma^2}}} \leq \br{\frac{\sigma^2}{\gamma^2 np}}^3,
\end{align*}
which is due to the fact $4\ebr{-\sqrt{x}/16}\leq 1/x^2$ when $x\geq x_0$ for some large $x_0>0$.

Combining the above results together, we have
\begin{align*}
\ell(\hat z,z^*) &\leq \frac{(1+\eta)\br{\frac{1}{1-c_2\frac{1}{\sqrt{np}} -\frac{1}{n}}}^2}{ \br{1-c_2\br{\sqrt{\frac{\log n}{np}} + \frac{1}{\log (np)}} - \gamma - 2\rho}^2}  \br{1+c_2\sqrt{\frac{\log n}{n}}} \frac{\sigma^2}{2np}+\br{\frac{\sigma^2}{\gamma^2 np}}^3\\
&\quad  +32 (1+\eta^{-1}+\rho^{-2}) c_2^2\br{\frac{2c_2}{\sqrt{np}} }^2  \frac{\sigma^2}{np}\\
&\quad +128(1+\eta^{-1} + \rho^{-2}) \br{1+\frac{c_2^2\sigma^2}{np}}c_2^2  \frac{\sigma^4+\sigma^2}{(np)^2}.
\end{align*}
Note that $\frac{1}{(1-x)^2}\leq 1 + 16x,\forall \leq x\leq \frac{1}{2}$. We have $\br{1-c_2\br{\sqrt{\frac{\log n}{np}} + \frac{1}{\log (np)}} - \gamma - 2\rho}^{-2}\leq 16\br{c_2\br{\sqrt{\frac{\log n}{np}} + \frac{1}{\log (np)}} + \gamma + 2\rho}$ and $\br{1-c_2\frac{1}{\sqrt{np}} -\frac{1}{n}}^{-2}\leq 16\br{c_2\frac{1}{\sqrt{np}} +\frac{1}{n}}$ as long as  $\frac{np}{\log n}$ is greater than some sufficiently large constant.
After rearrangement, there exists some constant $c_5>0$ such that
\begin{align*}
\ell(\hat z,z^*) \leq \Bigg(1+c_5\Bigg(&\eta + \gamma + \rho +\sqrt{\frac{\log n}{np}} + \frac{1}{\log (np)} + \gamma^{-6}\br{\frac{\sigma^2}{np}}^2 \\
&\quad+ (\eta^{-1} + \rho^{-2})\br{\frac{1+\sigma^2}{np}  }\Bigg) \Bigg)\frac{\sigma^2}{2np}.
\end{align*}
We can choose $\gamma^2 = \sqrt{{\sigma^2}/{(np)}}$ (then $\frac{\gamma^2 np}{\sigma^2}>c_3$ is guaranteed as long as $\frac{np}{\sigma^2}>c_3^2$). We also set $\rho^2 = \sqrt{(1+\sigma^2)/np}$ and let $\eta =\rho^2$. Then, there exists some constant $c_6>0$ such that
\begin{align*}
\ell(\hat z,z^*)\leq \br{1+c_6\br{\br{\frac{\sigma^2}{np}}^\frac{1}{4} +\sqrt{\frac{\log n}{np}}  + \frac{1}{\log (np)}}}\frac{\sigma^2}{2np}.
\end{align*}
This holds with probability at least $1-n^{-9} -\ebr{-\frac{1}{32}\br{\frac{np}{\sigma^2}}^\frac{1}{4}}$.
\end{proof}

\subsection{Proofs of Auxiliary Lemmas}

\begin{proof}[Proof of Lemma \ref{lem:davis}]
Let $\tilde \lambda_1\geq \tilde \lambda_2\geq \ldots \geq \tilde \lambda_d$ be eigenvalues of $\tilde X$. By Weyl's inequality, we have $\| \tilde \lambda_{r+1}  -\lambda_{r+1}\| \leq \|X-\tilde X\|$. Under the assumption $\|{X-\tilde X}\|< (\lambda_r-\lambda_{r+1})/4$, we have
\begin{align*}
\lambda_r -\tilde \lambda_{r+1} &= \lambda_r  - \lambda_{r+1} + \lambda_{r+1} - \tilde \lambda_{r+1}\geq \lambda_r  - \lambda_{r+1} - \norm{X -\tilde X} > \frac{3}{4} \br{\lambda_r  - \lambda_{r+1}} >0.
\end{align*}
Define 
\begin{align*}
\Theta(U,\tilde U) := \text{diag}( \cos^{-1}\sigma_1,\ldots,  \cos^{-1} \sigma_{r}) \in \mathr^{r\times r}, 
\end{align*}
where $\sigma_1 \geq \sigma_2\geq \ldots \geq \sigma_{r}$ are singular values of $U^\H \tilde U$. Since $\lambda_r -\tilde \lambda_{r+1}>0$, 
by  Davis-Kahan Theorem \cite{davis1970rotation}, we have 
\begin{align*}
\norm{\sin \Theta (U,\tilde U)}\leq \frac{\norm{X-\tilde X}}{\lambda_r -\tilde \lambda_{r+1}} \leq \frac{4\norm{X-\tilde X}}{3(\lambda_r - \lambda_{r+1})}.
\end{align*}
From \cite{davis1970rotation}, we also have $\|{\sin \Theta (U,\tilde U)}\| = \|{(I-UU^\H)\tilde U}\|$. The proof is complete.
\end{proof}

\begin{proof}[Proof of Lemma \ref{lem:complex_unit_vector_inequality}]
Since both $x$ and $y$ are unit vectors, we have
\begin{align}\label{eqn:rev_1}
\norm{x - y b}^2 =2 - x^\H y b - ( y b)^\H x =2-2\re(x^\H y b), \forall b\in\mathc_1.
\end{align}
 Therefore, when $x^\H y= 0$, we have $\norm{x - y b}=\sqrt{2}$ invariant of $b$. In this case, we also have  $\norm{(I_n-xx^\H)y}=\norm{y}=1$. This proves the  statement in the lemma  for the $x^\H y= 0$ case. When
$x^\H y\neq 0$, the infimum over $b$ in (\ref{eqn:rev_1}) is achieved when $b = y^\H x /|y^\H x|$. We then have
\begin{align*}
\inf_{b\in\mathc_1}\norm{x - y b}^2 & = \norm{y  - \frac{x^\H y}{\abs{x^\H y}} x}^2 = \norm{y  - xx^\H y + xx^\H y - \frac{x^\H y}{\abs{x^\H y}} x}^2 \\
&= \norm{y  - xx^\H y}^2 +  \norm{\br{1 - \frac{1}{\abs{x^\H y}}} (x^\H y)x}^2\\
&= \norm{y  - xx^\H y}^2 + \abs{1 - \frac{1}{\abs{x^\H y}}}^2\abs{x^\H y}^2\\
&\leq \norm{y  - xx^\H y}^2  + \abs{1-\abs{x^\H y}}^2,
\end{align*}
where we use the orthogonality between $(I_d -xx^\H)y$ and $x$.
With $\norm{y  - xx^\H y}^2  = 1+ \norm{xx^\H y}^2 - 2y^\H xx^\H y = 1- \abs{x^\H y}^2 \geq \br{1-\abs{x^\H y}}^2$, where the last inequality is due to $0\leq \abs{x^\H y}\leq 1$, the proof is complete.
\end{proof}

\begin{proof}[Proof of Lemma \ref{lem:A_related}]
Note that $\E A = pJ_n -pI_n$. Note that $(\one_n/\sqrt{n})^\T \E A (\one_n/\sqrt{n}) = (n-1)p$ and for any unit vector $u\in\mathr^n$ that is orthogonal to $\one_n/\sqrt{n}$, we have $u^\T \E A u = 0- p\|u\|^2 = -p$. Hence, $(n-1)p$ is the largest eigenvalue with $\one_n/\sqrt{n}$ being the corresponding eigenvector, and $-p$ is another eigenvalue with multiplicity $n-1$.

By Weyl's inequality, we have $|\lambda'-(n-1)p|,\max_{2\leq j\leq n}|\lambda'_j - (-p)| \leq \norm{A-\E A}$, which leads to (\ref{eqn:rev_9}) after rearrangement. This completes the proof,  with $\lambda^*=\lambda'$ and $\lambda^*_2=\lambda'_2$ by Lemma \ref{lem:no_additive_noise}.
\end{proof}

\begin{proof}[Proof of Lemma \ref{lem:combine}]
The first two inequalities stem from Lemma 5 and Lemma 6 of \cite{gao2021exact}, respectively. The third inequality is derived from Lemma 7 and (29) in  \cite{gao2021exact}.
\end{proof}

\begin{proof}[Proof of Lemma \ref{lem:re_exponential}]
It is proved in (31) of \cite{gao2021exact}.
\end{proof}

\section{Proof of Lemma \ref{lem:eigenspace_perturbation}}\label{sec:proof_od}

Before the proof, we  first state a technical lemma that is analogous to Lemma \ref{lem:complex_unit_vector_inequality}.
\begin{lemma}\label{lem:U_V_relationship}
For any two matrices $U,V\in\matho(d_1,d_2)$, we have
\begin{align*}
\norm{(I_{d_1} -  V V^\T)U}\leq \inf_{O\in \mathcal{O}(d_2)}\norm{V - U O} \leq \sqrt{2}\norm{(I_{d_1} - V V^\T)U}.
\end{align*}
\end{lemma}
\begin{proof}%
Let $V_\perp \in\mathr^{d_1\times (d_1-d_2)}$ be the complement of $V$ such that $(V,V_\perp) \in\matho(d_1)$. From  Lemma 1 of \cite{cai2018rate}, we have $\|U^\T V_\perp\| \leq  \inf_{O\in \mathcal{O}(d_2)}\norm{V - U O} \leq \sqrt{2} \|U^\T V_\perp\|$. The proof is complete with $ \|U^\T V_\perp\| = \|V_\perp V_\perp^\T U\| =\norm{(I_{d_1} - V V^\T)U}$.
\end{proof}

\begin{proof}[Proof of Lemma \ref{lem:eigenspace_perturbation}]
We first give an explicit expression for the first-order approximation $\tilde V$. Denote $\mu_1\geq \ldots\geq \mu_n$ as the eigenvalues of $Y$. 
Let $YV^* =  G  D  N^\T$ be its SVD where $ G \in \matho(n,d)$, $ N \in\matho(d)$, and $ D\in\mathr^{d\times d}$ is a diagonal matrix with singular values. Define $M^*=\text{diag}(\mu_1^*,\ldots,\mu^*_d)\in\mathr^{d\times d}$. Since 
\begin{align}\label{eqn:od_new_proof_1}
YV^* = Y^* V^* + (Y-Y^*)V^* =  V^*M^* + (Y-Y^*)V^*,
\end{align}
we have
\begin{align}\label{eqn:od_new_proof_3}
\max_{i\in[d]} \abs{D_{ii} - \mu^*_i}\leq \norm{ (Y-Y^*)V^*}\leq \norm{Y-Y^*},
\end{align}
by Weyl's inequality. Under the assumption that $\norm{Y-Y^*}\leq \min\{ \mu_d^*-\mu_{d+1}^*,\mu_d^*\}/4$, we have $\{D_{ii}\}_{i\in[d]}$ all being positive.
Note that
\begin{align*}
\tilde V &=\argmin_{V'\in \matho(n,d)}\fnorm{V'- YV^*}^2=\argmax_{V\in \matho(n,d)} \iprod{V'}{ YV^*} \\
&= \argmax_{V'\in \matho(n,d)} \text{tr}\br{V'^\T GDN^\T}= \argmax_{V'\in \matho(n,d)}\iprod{G^\T V' N}{ D}.
\end{align*}
 Due to the fact that $G, V' \in\matho(n, d)$, $N\in\matho(d)$, and the diagonal entries of $D$ are all positive,  the maximum is achieved when $G^\T V' N=I_d$. This gives $\tilde V = GN^\T$ which can also be written as
\begin{align}\label{eqn:od_proof_6}
\tilde V= YV^* S,
\end{align}
where
\begin{align}\label{eqn:od_proof_7}
S:= ND^{-1}N^\T\in\mathr^{d\times d}
\end{align}
 can be seen as a linear operator and plays a similar role as $1/\norm{Xu^*}$ for $\tilde u = Xu^*/\norm{Xu^*}$ in (\ref{eqn:tilde_u_def_intro}).

 Define $M:= \text{diag}(\mu_1,\mu_2,\ldots,\mu_d)\in\mathr^{d\times d}$. Then we have
\begin{align*}
VM&=YV,\\
\tilde VM & = YV^*SM,
\end{align*}
and consequently,
\begin{align*}
(V-\tilde V)M = Y(V- V^*SM) = Y(V- \tilde V) + Y(\tilde V - V^*SM).
\end{align*}
Note that $(I-VV^\T)Y=Y(I-VV^\T)$ as $V$ is the leading eigenspace of $Y$.
After rearranging, we have
\begin{align*}
Y\tilde V -\tilde V M = Y(\tilde V - V^*SM).
\end{align*}
Multiplying $(I-VV^\T)$ on both sides, we have
\begin{align*}
Y  (I-VV^\T)\tilde V - (I-VV^\T)\tilde V M &=(I-VV^\T)Y  \tilde V - (I-VV^\T)\tilde V M \\
&=  (I-VV^\T)Y  ({\tilde V- V^* S M}),
\end{align*}
where the first equation is due to $Y  (I-VV^\T) = (I-VV^\T)Y$ as $V$ is the leading eigenspace of $Y$.
Note that for any $x\in\text{span}(I-VV^\T)$ and for any $i\in[d]$, we have $\norm{Y  x - \mu_i x}\geq (\mu_i - \mu_{d+1})\norm{x}$. Then we have
\begin{align*}
\norm{Y  (I-VV^\T)\tilde V - (I-VV^\T)\tilde V M}\geq (\mu_d-\mu_{d+1}) \norm{ (I-VV^\T)\tilde V}.
\end{align*}
As a result, we have
\begin{align}\label{eqn:od_proof_5}
 \norm{ (I-VV^\T)\tilde V} \leq \frac{1}{\mu_d-\mu_{d+1}} \norm{ (I-VV^\T)Y  ({\tilde V- V^* S M})},
\end{align}
which is analogous to (\ref{eqn:proof_3}) in the proof of Lemma \ref{lem:eigenspace_perturbation}. 
By Lemma \ref{lem:U_V_relationship},
we have
\begin{align}\label{eqn:od_new_proof_2}
\inf_{O\in\matho(d)}\norm{V-\tilde VO}\leq \sqrt{2} \norm{ (I-VV^\T)\tilde V}  \leq  \frac{\sqrt{2}}{\mu_d-\mu_{d+1}} \norm{ (I-VV^\T)Y  ({\tilde V- V^* S M})}.
\end{align}

In the next, we are going to analyze $(I-VV^\T)Y  ({\tilde V- V^* S M})$. Using  (\ref{eqn:od_proof_6}), we have
\begin{align*}
& (I-VV^\T)Y  ({\tilde V- V^* S M}) \\
 &=  (I-VV^\T)Y  \br{YV^*S- V^* S M} \\
 & =  (I-VV^\T)Y  \br{  V^*M^*S + (Y-Y^*)V^*S- V^* S M} \\
 &=  (I-VV^\T)Y  V^* \br{M^*S  -  S M} +   (I-VV^\T)Y (Y-Y^*)V^*S\\
 &=  (I-VV^\T)\br{ V^*M^* + (Y-Y^*)V^*}\br{M^*S  -  S M} \\
 &\quad +   (I-VV^\T)V^* M^* V^{*\T} (Y-Y^*)V^*S \\
 &\quad +   (I-VV^\T)(Y^* -  V^* M^* V^{*\T})(Y-Y^*)V^*S +  (I-VV^\T)(Y-Y^*) (Y-Y^*)V^*S\\
 &=  (I-VV^\T) V^*M^* \br{\br{M^*S  -  S M} + V^{*\T} (Y-Y^*)V^*S}  \\
 &\quad + (I-VV^\T)(Y-Y^*)V^* \br{M^*S  -  S M} \\
 &\quad +   (I-VV^\T)(Y^* -  V^* M^* V^{*\T})(Y-Y^*)V^*S +  (I-VV^\T)(Y-Y^*) (Y-Y^*)V^*S,
\end{align*}
where in the second to  last equation, we use  (\ref{eqn:od_new_proof_1}) and the decomposition $Y=V^* M^* V^{*\T} + (Y^*-V^* M^* V^{*\T}) + (Y-Y^*)$. Hence, with $\|Y^* -  V^* M^* V^{*\T}\| = \max\{|\mu^*_{d+1}|,|\mu^*_n|\}$, we have 
\begin{align*}
&\norm{ (I-VV^\T)Y  ({\tilde V- V^* S M})} \\
&\leq \mu_1^*\norm{ (I-VV^\T) V^*} \br{\norm{M^*S  -  S M} + \norm{Y-Y^*}\norm{S}} \\
&\quad  + \norm{Y-Y^*}\norm{M^*S  -  S M}  + \max\{|\mu^*_{d+1}|,|\mu^*_n|\} \norm{Y-Y^*}\norm{S} +\norm{Y-Y^*}^2\norm{S}.
\end{align*}
Then from (\ref{eqn:od_new_proof_2}), we have
\begin{align*}
\inf_{O\in\matho(d)}\norm{V-\tilde VO} &\leq   \frac{\sqrt{2}}{\mu_d-\mu_{d+1}} \Bigg( \mu_1^*\norm{ (I-VV^\T) V^*} \br{\norm{M^*S  -  S M} + \norm{Y-Y^*}\norm{S}} \\
&\quad + \norm{Y-Y^*}\norm{M^*S  -  S M}   + \max\{|\mu^*_{d+1}|,|\mu^*_n|\} \norm{Y-Y^*}\norm{S} \\
&\quad+\norm{Y-Y^*}^2\norm{S}\Bigg).
\end{align*}

In the rest of the proof, we are going to simplify the  display above.  
By Weyl's inequality, we have
\begin{align}\label{eqn:rev_6}
\max_{i\in[n]}\abs{\mu_i-\mu_i^*}\leq \norm{Y-Y^*}.
\end{align}
Since $\norm{Y-Y^*}\leq  (\mu_d^*-\mu_{d+1}^*)/4$ is assumed, we have
\begin{align*}
\mu_d-\mu_{d+1} \geq \frac{\mu_d^* -\mu_{d+1}^*}{2}.
\end{align*}
By this assumption and Lemma \ref{lem:davis},
we have
\begin{align*}
\norm{ (I-VV^\T) V^*} 
 \leq \frac{2\norm{Y-Y^*}}{\mu^*_d -\mu^*_{d+1}}.
\end{align*}
By (\ref{eqn:od_new_proof_3}) and the definition of $S$ in (\ref{eqn:od_proof_7}), we have
\begin{align*}
\norm{S} = \norm{D^{-1}}\leq \frac{1}{\mu_d^* - \norm{Y-Y^*}} \leq \frac{4}{3\mu_d^*}.
\end{align*}
In addition,
\begin{align*}
\norm{M^*S  -  S M} &\leq \norm{M^*S  -  S M^*} + \norm{S\br{M -M^*}}\\
&\leq   \norm{(M^*-\mu_d^*I_d)S  +  S (\mu_d^*I_d-M^*)} + \norm{S} \norm{M-M^*}\\
&\leq \norm{S} \br{2\norm{M^*-\mu_d^*I_d} +  \norm{M-M^*}}\\
&\leq  \frac{4}{3\mu_d^*} \br{2(\mu_1^* - \mu_d^*) + \norm{Y-Y^*}},
\end{align*}
where in the last inequality we use the fact  $\norm{M-M^*} =\max_{i\in[d]}\abs{\mu_i - \mu_i^*}$ and (\ref{eqn:rev_6}). Combining all the results together, we have
\begin{align*}
&\inf_{O\in\matho(d)}\norm{V-\tilde VO} \\
&\leq \frac{2\sqrt{2}}{\mu_d^*-\mu_{d+1}^*} \br{\mu_1^*\frac{2\norm{Y-Y^*}}{\mu^*_d -\mu^*_{d+1}}\br{ \frac{4 \br{2(\mu_1^* - \mu_d^*) + \norm{Y-Y^*}}}{3\mu_d^*} + \frac{4\norm{Y-Y^*}}{3\mu_d^*}}}\\
&\quad +\frac{4}{3\mu_d^*} \br{2(\mu_1^* - \mu_d^*) + \norm{Y-Y^*}}\norm{Y-Y^*}  + \frac{4\max\{|\mu^*_{d+1}|,|\mu^*_n|\} \norm{Y-Y^*}}{3\mu_d^*} \\
&\quad +  \frac{4\norm{Y-Y^*}^2}{3\mu_d^*}\Bigg)\\
&\leq  \frac{16\sqrt{2}}{3\br{\mu_d^* - \mu_{d+1}^*}\mu_d^*}\br{\frac{2\mu_1^*}{3(\mu_d^*-\mu_{d+1}^*)} + 1}\norm{Y-Y^*}^2 \\
&\quad + \frac{8\sqrt{2}}{3\br{\mu_d^* - \mu_{d+1}^*}\mu_d^*} \br{\frac{4\mu_1^*\br{\mu_1^*-\mu_d^*}}{\mu_d^* - \mu_{d+1}^*} + 2(\mu_1^* - \mu_d^*) +  \max\{|\mu^*_{d+1}|,|\mu^*_n|\}} \norm{Y-Y^*}.
\end{align*}
\end{proof}

\bibliographystyle{plain}
\bibliography{phase}

\appendix

\section{Proofs of Lemma  \ref{lem:no_additive_noise_od}, Proposition \ref{prop:no_additive_noise_od}, and Proposition \ref{prop:eigenspace_perturbation}}\label{sec:appendix_A}

\begin{proof}[Proof of Lemma \ref{lem:no_additive_noise_od}]
Similar to the proof of Lemma \ref{lem:no_additive_noise}, we can show each eigenvalue of $A$ is also an eigenvalue of $(A\otimes J_d)\circ Z^*Z^{*\T}$ with multiplicity $d$. At the same time, each eigenvalue of $(A\otimes J_d)\circ Z^*Z^{*\T}$ must be an  eigenvalue of $A$. The proof is omitted here.
\end{proof}

\begin{proof}[Proof of Proposition \ref{prop:no_additive_noise_od}]
Since $\sigma=0$, we have $U=U^*$. Then $\hat Z_j = \mathp(U_j) = \mathp(U^*_j) = \mathp(Z^*_j \check u_j)$. Since $Z^*_j$ is an orthogonal matrix, we have $\hat Z_j = Z^*_j \text{sign}(\check u_j)$. Then by (\ref{eqn:no_additive_noise}), the proposition is proved by the same argument used to prove Proposition \ref{prop:no_additive_noise}.
\end{proof}

Before proving Proposition \ref{prop:eigenspace_perturbation}, we state some properties of $A$ and $\mathw$. The following lemma can be seen as an analog of Lemma \ref{lem:combine}.
\begin{lemma}\label{lem:od_operator_norm_W}
There exist constants $C_1,C_2>0$ such that if $\frac{np}{\log n}>C_1$, then we have
\begin{align*}
&\norm{(A\otimes J_d)\circ \mathw}\leq C_2\sqrt{dnp},\\
&\sum_{i=1}^n\fnorm{\sum_{j\in[n]\backslash\{i\}}A_{ij}\left(Z_i^{*\T}\mathw_{ij}Z_j^*-Z_j^{*\T}\mathw_{ji}Z_i^*\right)}^2 \leq 2d(d-1)n^2p\left(1+C_2\sqrt{\frac{\log n}{n}}\right), \\
& \sum_{i=1}^n\fnorm{\sum_{j\in[n]\backslash\{i\}}A_{ij}\mathw_{ij}Z_j^*}^2 \leq d^2n^2p\left(1+C_2\sqrt{\frac{\log n}{n}}\right),
\end{align*}
hold with probability at least $1-3n^{-10}$.
\end{lemma}
\begin{proof}%
The first inequality is from Lemma 4.2 of \cite{gao2021optimal}. The second and third inequalities are from (59) and (60), together with Lemma 4.3, of  \cite{gao2021optimal}, respectively.
\end{proof}

\begin{proof}[Proof of Proposition \ref{prop:eigenspace_perturbation}]
By Lemma \ref{lem:combine} and Lemma \ref{lem:od_operator_norm_W},  there exist constants $c_1,c_2>0$ such that when $\frac{np}{\log n}>c_1$, we have $\norm{A-\E A}\leq c_2\sqrt{np}$ and $\norm{(A\otimes J_d)\circ \mathw}\leq c_2\sqrt{dnp}$ with probability at least $1-6n^{-10}$. By Lemma \ref{lem:no_additive_noise_od} and Lemma \ref{lem:A_related}, we have $\lambda^*_1=\lambda^*_d \geq (n-1)p - c_2\sqrt{np}$, $\max\{|\lambda^*_{d+1}|,|\lambda^*_n|\}\leq p + c_2\sqrt{np}$, and  $\lambda^*_d -\lambda_{d+1}^*\geq np - 2c_2\sqrt{np}$. Note that $d$ is a constant. When $\frac{np}{\log n}$ and $\frac{np}{\sigma^2}$ are greater than some sufficiently large constant, we have $4\sigma \norm{(A\otimes J_d)\circ \mathw} \leq np/2 \leq  \min\{\lambda^*_d,\lambda^*_d -\lambda_{d+1}^*\}$ satisfied. Since $\mathx - (A\otimes J_d)\circ Z^*Z^{*\H} = \sigma (A\otimes J_d) \circ \mathw$, a direct application of Lemma \ref{lem:eigenspace_perturbation} leads to
\begin{align*}
&\inf_{O\in\matho(d)} \norm{U-\tilde UO} \\
&\leq \frac{8\sqrt{2}}{3(\lambda_1^* -\lambda_{d+1}^*)} \Bigg(\br{\frac{4}{3(\lambda_1^* - \lambda_{d+1}^*)}+\frac{2}{\lambda_1^*}} \sigma^2\norm{(A\otimes J_d) \circ \mathw}^2  \\
&\quad + \frac{\max\{|\lambda^*_{d+1}|,|\lambda^*_n|\}}{\lambda_1^*}\sigma\norm{(A\otimes J_d) \circ \mathw}\Bigg)\\
& = \frac{8\sqrt{2}}{3(np/2)} \br{\br{\frac{4}{3(np/2)}+\frac{2}{np/2}} \sigma^2c_2^2dnp + \frac{p+c_2\sqrt{np}}{np/2}\sigma c_2\sqrt{dnp}}\\
&\leq c_3 \frac{\sigma^2 d + \sigma \sqrt{d}}{np},
\end{align*} 
for some constant $c_3>0$.
\end{proof}

\section{Proof of Theorem \ref{thm:od}}\label{sec:proof_od_thm}

We first state useful technical lemmas. They are analogs of Lemma \ref{lem:re_exponential} and Lemma \ref{lem:x_normalize_y_diff}, respectively. Lemma \ref{lem:od_re_exponential}  is proved in (31) of \cite{gao2021optimal}.

\begin{lemma}\label{lem:od_re_exponential}
There exists some constant $C > 0$ such that for any $\rho$ that satisfies $\frac{\rho^2np}{d^2\sigma^2} \geq  C$ , we
\begin{align*}
\sum_{i=1}^n\mathbb{I}\left\{\frac{2\sigma}{np}\norm{\sum_{j\in[n]\backslash\{i\}}A_{ij}\mathw_{ij}Z_j^*}>\rho\right\} \leq \frac{\sigma^2}{\rho^2p}\exp\left(-\sqrt{\frac{\rho^2 np}{\sigma^2}}\right), 
\end{align*}
with probability at least $1-\ebr{-\sqrt{\frac{\rho^2 np}{\sigma^2}}}$.
\end{lemma}

\begin{lemma}[Lemma 2.1 of \cite{gao2021optimal}]\label{lem:od_normalization}
Let $X,\tilde{X}\in\mathbb{R}^{d\times d}$ be two matrices of full rank. Then,
$$\fnorm{\mathcal{P}(X)-\mathcal{P}(\tilde{X})}\leq \frac{2}{s_{\min}(X)+s_{\min}(\tilde{X})}\fnorm{X-\tilde{X}}.$$
\end{lemma}

\begin{proof}[Proof of Theorem \ref{thm:od}]
Let  $O\in\matho(d)$ satisfy $\|U-\tilde UO\|=\inf_{O'\in\matho(d)}\|{U-\tilde UO'}\|$. Define $\Delta := U-\tilde UO\in\mathr^{nd\times d}$.  Recall $\check u$ is
the leading eigenvector of $A$. From Proposition \ref{prop:no_additive_noise}, Proposition \ref{prop:eigenspace_perturbation}, Lemma \ref{lem:combine}, and Lemma \ref{lem:od_operator_norm_W},
there exist constants $c_1,c_2>0$ such that if $\frac{np}{\log n},\frac{np}{\sigma^2}>c_1$, we have
\begin{align}
\norm{\Delta} &\leq c_2 \frac{\sigma^2d+\sigma\sqrt{d}}{np}, \label{eqn:od_new_proof_4}\\
\max_{j\in[n]}\abs{\check u_j - \frac{1}{\sqrt{n}}b_2}&\leq c_2\br{\sqrt{\frac{\log n}{np}} + \frac{1}{\log (np)}}\frac{1}{\sqrt{n}},\label{eqn:od_new_proof_5}\\
\norm{A-\E A} &\leq c_2\sqrt{np},\label{eqn:od_new_proof_6}\\
\norm{(A\otimes J_d)\circ \mathw}&\leq c_2\sqrt{npd},\label{eqn:od_new_proof_7}\\
\sum_{i=1}^n\fnorm{\sum_{j\in[n]\backslash\{i\}}A_{ij}\left(Z_i^{*\T}\mathw_{ij}Z_j^*-Z_j^{*\T}\mathw_{ji}Z_i^*\right)}^2 &\leq 2d(d-1)n^2p\left(1+c_2\sqrt{\frac{\log n}{n}}\right), \label{eqn:rev_12}\\
 \sum_{i=1}^n\fnorm{\sum_{j\in[n]\backslash\{i\}}A_{ij}\mathw_{ij}Z_j^*}^2 &\leq d^2n^2p\left(1+c_2\sqrt{\frac{\log n}{n}}\right),\label{eqn:rev_13}
\end{align}
with probability at least $1-n^{-9}$, for some $b_2\in\{-1,1\}$. By Lemma \ref{lem:no_additive_noise_od} and Lemma \ref{lem:A_related}, we have $\lambda^*_1=\lambda^*_d$, $|\lambda_d^* -  (n-1)p|\leq  c_2\sqrt{np}$, $\abs{\lambda^*_{d+1}}\leq p + c_2\sqrt{np}$, and  $\lambda^*_d -\lambda_{d+1}^*\geq np - 2c_2\sqrt{np}$.

Using the same argument as (\ref{eqn:od_proof_6}) and (\ref{eqn:od_proof_7}) in the proof of Lemma \ref{lem:eigenspace_perturbation}, we can have an explicit expression for $\tilde U$. Recall the definition of $\tilde U$ in (\ref{eqn:tilde_U_def}).  Let $\mathx U^* =GDN^\T$ be its SVD where $ G \in \matho(nd,d)$, $ N \in\matho(d)$, and $ D\in\mathr^{d\times d}$ is a diagonal matrix with singular values. By the decomposition (\ref{eqn:mathX_matrix_form}), we have
\begin{align}\label{eqn:od_new_proof_8}
\mathx U^* = ((A\otimes J_d)\circ Z^*Z^{*\T}) U^* + \sigma ((A\otimes J_d)\circ\mathw)U^*= \lambda_1^*U^* + \sigma ((A\otimes J_d)\circ\mathw)U^*.
\end{align}
Since the diagonal entries of $D$ correspond to the leading singular values of $\mathx U^*$, Weyl's inequality leads to $\max_{j\in[d]}|D_{jj} - \lambda_1^*| \leq  \sigma \norm{(A\otimes J_d)\circ\mathw} \leq c_2\sigma \sqrt{dnp}.$ Denote
\begin{align}\label{eqn:rev_17}
t :=  p + c_2\sqrt{np} +c_2\sigma\sqrt{d np}.
\end{align}
We then have
\begin{align}\label{eqn:od_new_proof_12}
\max_{j\in[d]}|D_{jj} - np| \leq  p + t.
\end{align}
When $\frac{np}{\log n},\frac{np}{d\sigma^2}$ are greater than some sufficiently large constant, we have $np/2 \leq \lambda^*_1$ and $np/2 \leq D_{jj}\leq 3np/2$ for all $j\in[d]$. As a consequence, all the diagonal entries of $D$ are positive.
Then $\tilde U$ can be written as
\begin{align*}
\tilde U = \mathx U^* S,
\end{align*}
where
\begin{align}\label{eqn:od_new_proof_9}
S:= ND^{-1}N^\T\in\mathr^{d\times d}.
\end{align}
Then (\ref{eqn:od_new_proof_12}) leads to
\begin{align}\label{eqn:rev_7}
 \norm{\frac{1}{np}I_d - S} = \norm{\frac{1}{np} I_d -D^{-1}} \leq \frac{1}{np -t} - \frac{1}{np}\leq \frac{2t}{(np)^2},
\end{align}
and
\begin{align}\label{eqn:rev_8}
\norm{S}=\norm{D^{-1}} \leq \frac{2}{np}.
\end{align}

Using (\ref{eqn:od_new_proof_8}), we have the following decomposition for $U$:
\begin{align*}
U = \tilde U O + \Delta = \mathx U^* SO  + \Delta = \br{\lambda_1^* U^* +  \sigma ((A\otimes J_d)\circ\mathw)U^*}SO + \Delta.
\end{align*}
Recall the definition of $U^*$ in (\ref{eqn:U_star_def}). Define $\Delta^*:= U^* - \frac{1}{\sqrt{n}}Z^*b_2$. When $\frac{np}{\log n}\geq 2c_2^*$, by the same argument used to derive (\ref{eqn:new_proof_10}) as in the proof of Theorem \ref{thm:phase_main}, we have
\begin{align}
\norm{\Delta^*} &= \norm{Z^*\circ\br{\check u\otimes \one_d - \frac{1}{\sqrt{n}}\one_n\otimes \one_d b_2}} = \norm{\check u\otimes \one_d - \frac{1}{\sqrt{n}}\one_n\otimes \one_d} =\sqrt{d} \norm{\check u -\frac{1}{\sqrt{n}} \one_n b_2}\nonumber\\
&\leq \frac{2c_2\sqrt{np} + 2p}{np} \sqrt{d}.\label{eqn:od_new_proof_13}
\end{align}
Then $U$ can be further decomposed into
\begin{align*}
U = \br{\lambda_1^* U^* +  \sigma ((A\otimes J_d)\circ\mathw)\br{ \frac{1}{\sqrt{n}}Z^*b_2 + \Delta^*}}SO + \Delta.
\end{align*}
For any $j\in[n]$, denote $[(A\otimes J_d)\circ\mathw]_{j\cdot}\in\mathr^{d\times nd}$ as the submatrix corresponding to its rows from the $((j-1)d+1)$th to the $(jd)$th. Note that  $SO\in\mathr^{d\times d}$. Then $U_j$ has an expression:
\begin{align*}
U_j &= \br{\lambda_1^* U_j^* +\frac{\sigma}{\sqrt{n}}  [(A\otimes J_d)\circ\mathw]_{j\cdot} Z^* b_2 + \sigma  [(A\otimes J_d)\circ\mathw]_{j\cdot}  \Delta^*} SO + \Delta_j\\
&=  \br{\lambda_1^* Z_j^* \check u_j +\frac{\sigma}{\sqrt{n}} \sum_{k\neq j} A_{jk}\mathw_{jk}Z^*_k b_2 + \sigma  [(A\otimes J_d)\circ\mathw]_{j\cdot}  \Delta^*} SO + \Delta_j,
\end{align*}
where $\Delta_j\in\mathr^{d\times d}$ is denoted as the $j$th submatrix of $\Delta$.

Note that we have  following properties for the mapping $\mathp$. For any $B\in\mathr^{d\times d}$ of full rank and any $F\in\matho(d)$, we have $\mathp(BF)=\mathp(B)F$. In addition, if $B$ is positive-definite, $\mathp(B)=I_d$. Since we have shown the diagonal entries of $D$ are all lower bounded by $np/2$, (\ref{eqn:od_new_proof_9}) leads to $\mathp(S)=I_d$. Then
\begin{align*}
\fnorm{\hat Z_j -Z^*_j O b_2}  = \fnorm{\mathp(U_j) -Z^*_j O b_2} = \fnorm{\mathp(Z_j^{*\T}U_j O^\T b_2) -I_d}.
\end{align*}
We have
\begin{align*}
Z_j^{*\T}U_j O^\T b_2 = \br{\lambda_1^* \check u_j  b_2 I_d + \frac{\sigma}{\sqrt{n}}\Xi_j  + \sigma  b_2Z_j^{*\T} [(A\otimes J_d)\circ\mathw]_{j\cdot}  \Delta^*} S + Z_j^{*\T}\Delta_j O^\T b_2
\end{align*}
where 
\begin{align*}
\Xi_j := \sum_{k\neq j}A_{jk} Z^{*\T}_j \mathw_{jk} Z^*_k.
\end{align*}
Note that  from (\ref{eqn:od_new_proof_5}), we have
\begin{align*}
 b_2 \check u_j \geq  \br{1-c_2\br{\sqrt{\frac{\log n}{np}} + \frac{1}{\log (np)}}}\frac{1}{\sqrt{n}}.
\end{align*}
As long as $\frac{np}{\log n}$ is greater than some sufficiently large constant,
we have
$
 b_2 \check u_j  \geq \frac{1}{2\sqrt{n}}.
$
Since $\lambda_1^*$ is also positive, we have
\begin{align}\label{eqn:rev_15}
\frac{Z_j^{*\T}U_j O^\T b_2}{\lambda_1^* \check u_j  b_2}= S  + T_j
\end{align}
where $T_j$ is defined as
\begin{align*}
T_j &:= \frac{1}{\lambda_1^* \check u_j  b_2} \br{\br{ \frac{\sigma}{\sqrt{n}}\Xi_j + \sigma b_2Z_j^{*\T} [(A\otimes J_d)\circ\mathw]_{j\cdot}  \Delta^*} S + Z_j^{*\T}\Delta_j O^\T b_2}\\
& = \frac{1}{\lambda_1^* \check u_jb_2 } \frac{\sigma}{\sqrt{n}}\Xi_j S  +  \frac{\sigma b_2Z_j^{*\T} [(A\otimes J_d)\circ\mathw]_{j\cdot}  \Delta^* S}{\lambda_1^* \check u_j  b_2}  +  \frac{ Z_j^{*\T}\Delta_j O^\T b_2}{\lambda_1^* \check u_j  b_2}.
\end{align*}
As a consequence, when $\det(U_j)\neq 0$, we have
\begin{align}\label{eqn:rev_16}
\fnorm{\hat Z_j -Z^*_j Ob_2}  = \fnorm{\mathp\br{\frac{Z_j^{*\T}U_j O^\T b_2}{\lambda_1^* \check u_j  b_2}} -I_d} &= \fnorm{\mathp\br{S+ T_j } - I_d}.
\end{align}

Let $0 < \gamma,\rho < 1/8$ whose values will be determined later.  To simplify $\|\hat Z_j -Z^*_j Ob_2\|_{\text{F}} $, consider the following two cases.

(1) If
\begin{align}\label{eqn:od_new_proof_10}
\norm{\frac{1}{\lambda_1^* \check u_jb_2 } \frac{\sigma}{\sqrt{n}}\Xi_j S } &\leq  \frac{\gamma}{np}\\
\norm{ \frac{\sigma b_2Z_j^{*\T} [(A\otimes J_d)\circ\mathw]_{j\cdot}  \Delta^* S}{\lambda_1^* \check u_j  b_2} } &\leq \frac{\rho}{np}\nonumber\\
\norm{  \frac{ Z_j^{*\T}\Delta_j O^\T b_2}{\lambda_1^* \check u_j  b_2}} &\leq \frac{\rho}{np}\label{eqn:od_new_proof_11}
\end{align}
all hold, then 
\begin{align*}
s_{\min}(S+ T_j) &\geq s_{\min}(S) - \norm{T_j}  =s_{\min}(D^{-1}) - \norm{T_j} = D_{11}^{-1}- \norm{T_j} \\
&\geq D_{11}^{-1} - \frac{\gamma + 2\rho}{np},
\end{align*}
which is greater than 0 by (\ref{eqn:od_new_proof_12}). Together with (\ref{eqn:rev_15}), we have $\det(U_j)\neq 0$.
The same lower bound holds for $s_{\min}(S+ (T_j+T_j^\T)/2)$. Since $S$ is positive-definite, we have  $\mathp(S+ (T_j+T_j^\T)/2)=I_d$.
By Lemma \ref{lem:od_normalization} and (\ref{eqn:rev_16}), we have
\begin{align*}
&\fnorm{\hat Z_j -Z^*_j O b_2} \\
& =\fnorm{\mathp\br{S+ T_j } - \mathp\br{S + \frac{T_j+T_j^\T}{2}}} \\
&\leq \frac{1}{\br{D_{11}^{-1} - \frac{\gamma+2\rho}{np}}}\fnorm{\frac{T_j - T_j^\T}{2}}\\
&\leq \frac{1}{\lambda_1^* \check u_jb_2 } \frac{1}{2\br{D_{11}^{-1} - \frac{\gamma+2\rho}{np}}} \Bigg( \frac{\sigma}{\sqrt{n}}\fnorm{\Xi_j S - S^\T \Xi_j^\T} +2\fnorm{\sigma b_2Z_j^{*\T} [(A\otimes J_d)\circ\mathw]_{j\cdot}  \Delta^* S} \\
&\quad + 2\fnorm{ Z_j^{*\T}\Delta_j O^\T b_2}  \Bigg).
\end{align*}
We can further simplify the first term in the  display above. We have
\begin{align*}
\fnorm{\Xi_j S - S^\T \Xi_j^\T}& = \fnorm{\frac{1}{np}\br{\Xi_j - \Xi_j^\T} -\Xi_j\br{\frac{1}{np}I_d- S } +(\frac{1}{np}I_d - S^\T )\Xi_j^\T  } \\
&\leq  \frac{1}{np} \fnorm{\Xi_j - \Xi_j^\T} + 2\norm{\frac{1}{np}I_d - S}\fnorm{\Xi_j}.
\end{align*}
Using (\ref{eqn:rev_7}) and (\ref{eqn:rev_8}),
we have
\begin{align*}
\fnorm{\hat Z_j -Z^*_j  O b_2}  &\leq  \frac{1}{\lambda_1^* \check u_jb_2 } \frac{1}{2\br{D_{11}^{-1} - \frac{\gamma+2\rho}{np}}} \Bigg( \frac{\sigma}{\sqrt{n}} \frac{1}{np} \fnorm{\Xi_j - \Xi_j^\T}  +\frac{\sigma}{\sqrt{n}}  \frac{t}{(np)^2}\fnorm{\Xi_j} \\
&\quad + \frac{4}{np}\sigma \fnorm{ [(A\otimes J_d)\circ\mathw]_{j\cdot}  \Delta^*} +2\fnorm{ \Delta_j}  \Bigg).
\end{align*}
Using the lower bounds for $\lambda_1^*$, $\check u_jb_2$, and $D_{11}^{-1}$, as given at the beginning of this proof, we have
\begin{align*}
&\fnorm{\hat Z_j -Z^*_j  O b_2}  \\
&\leq \frac{1}{\br{np-p-c_2\sqrt{np}} \br{1-c_2\br{\sqrt{\frac{\log n}{np}} + \frac{1}{\log (np)}}}\br{\frac{1}{np+t}- \frac{\gamma+2\rho}{np}}} \frac{\sigma}{2np}\fnorm{\Xi_j - \Xi_j^\T} \\
&\quad + \frac{4\sigma t}{(np)^2}\fnorm{\Xi_j}+ \frac{16\sigma\sqrt{n}}{np} \fnorm{ [(A\otimes J_d)\circ\mathw]_{j\cdot}  \Delta^*}  +16\sqrt{n} \fnorm{ \Delta_j}.
\end{align*}
Let $\eta>0$ whose value will be given later. By the same argument as used in the proof of Theorem \ref{thm:phase_main}, we have
\begin{align*}
&\fnorm{\hat Z_j -Z^*_j  O b_2}^2\\
&\leq \frac{1+\eta}{\br{np-p-c_2\sqrt{np}}^2 \br{1-c_2\br{\sqrt{\frac{\log n}{np}} + \frac{1}{\log (np)}}}^2\br{\frac{1}{np+t}- \frac{\gamma+2\rho}{np}}^2} \frac{\sigma^2}{4(np)^2}\fnorm{\Xi_j - \Xi_j^\T}^2 \\
&\quad + 3(1+\eta^{-1})\frac{16\sigma^2 t^2}{(np)^4}\fnorm{\Xi_j}^2+  3(1+\eta^{-1})\frac{256\sigma^2n}{(np)^2} \fnorm{ [(A\otimes J_d)\circ\mathw]_{j\cdot}  \Delta^*}^2 \\
&\quad  + 3(1+\eta^{-1})64n \fnorm{ \Delta_j}^2.
\end{align*}

(2) If any one of (\ref{eqn:od_new_proof_10})-(\ref{eqn:od_new_proof_11}) does not hold, we simply upper bound $\normf{\hat Z_j -Z^*_j \tilde Qb_2} $ by $2\sqrt{d}$. Then this case can be written as
\begin{align*}
&\fnorm{\hat Z_j -Z^*_j  O b_2}^2\\
&\leq 4d \Bigg( \indic{\norm{\frac{1}{\lambda_1^* \check u_jb_2 } \frac{\sigma}{\sqrt{n}}\Xi_j S } > \frac{\gamma}{np}} + \indic{\norm{ \frac{\sigma b_2Z_j^{*\T} [(A\otimes J_d)\circ\mathw]_{j\cdot}  \Delta^* S}{\lambda_1^* \check u_j  b_2} } > \frac{\rho}{np}} \\
&\quad + \indic{\norm{  \frac{ Z_j^{*\T}\Delta_j O^\T b_2}{\lambda_1^* \check u_j  b_2}} > \frac{\rho}{np}}\Bigg).
\end{align*}
Using (\ref{eqn:rev_8}), $\lambda_1^*\geq np/2$, and $\check u_j b_2\geq 1/(2\sqrt{n})$, we have
\begin{align*}
&\fnorm{\hat Z_j -Z^*_j  O b_2}^2 \\
&\leq 4d \br{ \indic{ 8\sigma \norm{\Xi_j } \geq \gamma np} +\indic{8\sqrt{n}\sigma \norm{ [(A\otimes J_d)\circ\mathw]_{j\cdot}  \Delta^*} \geq \rho np} + \indic{4\sqrt{n}\norm{\Delta_j}\geq \rho}}\\
&\leq  4d \br{ \indic{ 8\sigma \norm{\Xi_j } \geq \gamma np} +  \frac{64\sigma^2 n}{(\rho np)^2}\fnorm{ [(A\otimes J_d)\circ\mathw]_{j\cdot}  \Delta^*}^2 + 16n\rho^{-2}\fnorm{\Delta_j}^2 }.
\end{align*} 

Combining these two cases together, we have
\begin{align*}
&\fnorm{\hat Z_j -Z^*_j  O b_2}^2 \\
&\leq  \frac{1+\eta}{\br{np-p-c_2\sqrt{np}}^2 \br{1-c_2\br{\sqrt{\frac{\log n}{np}} + \frac{1}{\log (np)}}}^2\br{\frac{1}{np+t}- \frac{\gamma+2\rho}{np}}^2} \frac{\sigma^2}{4(np)^2}\fnorm{\Xi_j - \Xi_j^\T}^2 \\
&\quad + 3(1+\eta^{-1})\frac{16\sigma^2 t^2}{(np)^4}\fnorm{\Xi_j}^2+  3(1+\eta^{-1})\frac{256\sigma^2n}{(np)^2} \fnorm{ [(A\otimes J_d)\circ\mathw]_{j\cdot}  \Delta^*}^2 \\
&\quad + 3(1+\eta^{-1})64n \fnorm{ \Delta_j}^2\\
&\quad + 4d \br{ \indic{ 8\sigma \norm{\Xi_j } \geq \gamma np} +  \frac{64\sigma^2 n}{(\rho np)^2}\fnorm{ [(A\otimes J_d)\circ\mathw]_{j\cdot}  \Delta^*}^2 + 16n\rho^{-2}\fnorm{\Delta_j}^2 } \\
&\leq \frac{1+\eta}{\br{np-p-c_2\sqrt{np}}^2 \br{1-c_2\br{\sqrt{\frac{\log n}{np}} + \frac{1}{\log (np)}}}^2\br{\frac{1}{np+t}- \frac{\gamma+2\rho}{np}}^2} \frac{\sigma^2}{4(np)^2}\fnorm{\Xi_j - \Xi_j^\T}^2 \\
&\quad + 3(1+\eta^{-1})\frac{16\sigma^2 t^2}{(np)^4}\fnorm{\Xi_j}^2 +  4d  \indic{ 8\sigma \norm{\Xi_j } \geq \gamma np} \\
&\quad + \frac{256\sigma^2n}{(np)^2}\br{3(1+\eta^{-1}) + d\rho^{-2}}\fnorm{ [(A\otimes J_d)\circ\mathw]_{j\cdot}  \Delta^*}^2 \\
&\quad+ 64n \br{3(1+\eta^{-1})+d\rho^{-2}}\fnorm{\Delta_j}^2.
\end{align*}
As a result, we have
\begin{align*}
&\ellod(\hat Z,Z^*)\\
&\leq  \frac{1}{n}\sum_{j\in[n]}\fnorm{\hat Z_j -Z^*_j  O b_2}^2 \\
&\leq  \frac{1+\eta}{\br{np-p-c_2\sqrt{np}}^2 \br{1-c_2\br{\sqrt{\frac{\log n}{np}} + \frac{1}{\log (np)}}}^2\br{\frac{1}{np+t}- \frac{\gamma+2\rho}{np}}^2} \\
&\quad\times \frac{\sigma^2}{4(np)^2} \frac{1}{n}\sum_{j\in[n]}\fnorm{\Xi_j - \Xi_j^\T}^2\\
&\quad + 3(1+\eta^{-1})\frac{16\sigma^2 t^2}{(np)^4} \frac{1}{n}\sum_{j\in[n]}\fnorm{\Xi_j}^2 +  4d   \frac{1}{n}\sum_{j\in[n]}\indic{ 8\sigma \norm{\Xi_j } \geq \gamma np} \\
&\quad + \frac{256\sigma^2}{(np)^2}\br{3(1+\eta^{-1}) + d\rho^{-2}} \sum_{j\in[n]}\fnorm{ [(A\otimes J_d)\circ\mathw]_{j\cdot}  \Delta^*}^2\\
&\quad + 64 \br{3(1+\eta^{-1})+d\rho^{-2}} \sum_{j\in[n]}\fnorm{\Delta_j}^2.
\end{align*}

In the rest of the proof, we are going to simplify the display above. Specifically, we are going to upper bound $\sum_{j\in[n]}\|{\Xi_j - \Xi_j^\T}\|_{\rm F}^2$, $\sum_{j\in[n]}\fnorm{\Xi_j}^2 $, $\sum_{j\in[n]}\indic{ 8\sigma \norm{\Xi_j } \geq \gamma np}$, $ \sum_{j\in[n]}\fnorm{ [(A\otimes J_d)\circ\mathw]_{j\cdot}  \Delta^*}^2$, and $\sum_{j\in[n]}\fnorm{\Delta_j}^2$.

For $\sum_{j\in[n]}\|{\Xi_j - \Xi_j^\T}\|_{\rm F}^2$ and $\sum_{j\in[n]}\fnorm{\Xi_j}^2 $, note that they are the left-hand sides of (\ref{eqn:rev_12}) and (\ref{eqn:rev_13}), respectively. Hence, they can be upper bounded by the right-hand sides of (\ref{eqn:rev_12}) and (\ref{eqn:rev_13}), respectively. For $\sum_{j\in[n]}\indic{ 8\sigma \norm{\Xi_j } \geq \gamma np}$, according to 
Lemma \ref{lem:od_re_exponential}, if  $\frac{\gamma^2 np}{d^2\sigma^2} >c_3$ for some $c_3>0$, we have
\begin{align*}
\sum_{j\in[n]}\indic{ 8\sigma \norm{\Xi_j } \geq \gamma np}  &\leq \frac{16\sigma^2}{\gamma^2 p}\ebr{-\sqrt{\frac{\gamma^2np}{16\sigma^2}}}
\end{align*}
with probability at least $1-\ebr{-\sqrt{\frac{\gamma^2np}{16\sigma^2}}}$. When $c_3$ is sufficiently large, it follows that 
\begin{align*}
\frac{16\sigma^2}{\gamma^2 np}\ebr{-\sqrt{\frac{\gamma^2np}{16\sigma^2}}} \leq \br{\frac{\sigma^2}{\gamma^2np}}^3
\end{align*}
by the same argument as in the proof of Theorem \ref{thm:phase_main}. For $ \sum_{j\in[n]}\fnorm{ [(A\otimes J_d)\circ\mathw]_{j\cdot}  \Delta^*}^2$,  we have
\begin{align*}
 \sum_{j\in[n]}\fnorm{ [(A\otimes J_d)\circ\mathw]_{j\cdot}  \Delta^*}^2&=\fnorm{ (A\otimes J_d)\circ\mathw  \Delta^*}^2 \\
 & \leq \norm{ (A\otimes J_d)\circ\mathw  }^2\fnorm{\Delta^*}^2 \\
 &\leq d\norm{ (A\otimes J_d)\circ\mathw  }^2\norm{\Delta^*}^2  \\
 & \leq c_2 d\br{\sqrt{dnp} \frac{2c_2\sqrt{np} + 2p}{np} \sqrt{d}}^2,
\end{align*}
where in the second to  last inequality we use the fact that $\Delta^*$ is rank-$d$ and in the last inequality we use (\ref{eqn:od_new_proof_13}). For $ \sum_{j\in[n]}\fnorm{\Delta_j}^2$, we have  $\sum_{j\in[n]}\fnorm{\Delta_j}^2 =\fnorm{\Delta}^2\leq d\norm{\Delta}^2\leq d\br{c_2 \frac{\sigma^2d+\sigma\sqrt{d}}{np}}^2$ where the last inequality is due to (\ref{eqn:od_new_proof_4}).

Using the above results, we have
\begin{align*}
&\ellod(\hat Z,Z^*)\\
&\leq  \frac{1+\eta}{\br{np-p-c_2\sqrt{np}}^2 \br{1-c_2\br{\sqrt{\frac{\log n}{np}} + \frac{1}{\log (np)}}}^2\br{\frac{1}{np+t}- \frac{\gamma+2\rho}{np}}^2} \\
&\quad \times \frac{\sigma^2}{4(np)^2} 2d(d-1)np\left(1+c_2'\sqrt{\frac{\log n}{n}}\right)\\
&\quad + 3(1+\eta^{-1})\frac{16\sigma^2 t^2}{(np)^4} d^2np\left(1+c_2'\sqrt{\frac{\log n}{n}}\right)+  4d   \br{\frac{\sigma^2}{\gamma^2np}}^3\\
&\quad + \frac{256\sigma^2}{(np)^2}\br{3(1+\eta^{-1}) + d\rho^{-2}} c_2d\br{\sqrt{dnp} \frac{2c_2\sqrt{np} + 2p}{np} \sqrt{d}}^2 \\
&\quad+ 64 \br{3(1+\eta^{-1})+d\rho^{-2}}d\br{c_2 \frac{\sigma^2d+\sigma\sqrt{d}}{np}}^2.
\end{align*}
Note that $\frac{1}{(1-x)^2}\leq 1 + 16x$ for any $0\leq x\leq \frac{1}{2}$. When $\frac{np}{\log n}$ is greater than some sufficiently large constant, we have $\br{1-c_2\br{\sqrt{\frac{\log n}{np}} + \frac{1}{\log (np)}}}^{-2}\leq 16c_2\br{\sqrt{\frac{\log n}{np}} + \frac{1}{\log (np)}} $ and $\br{1-c_2\frac{1}{\sqrt{np}} -\frac{1}{n}}^{-2}\leq 16\br{c_2\frac{1}{\sqrt{np}} +\frac{1}{n}}$. When $\frac{np}{d\sigma^2}$ is also greater than some sufficiently large constant, we have $\br{\frac{np}{np+t} - \gamma -2\rho}^{-2}\leq 16\br{\frac{t}{np+t} + \gamma +2\rho}\leq 16\br{\frac{t}{np} + \gamma +2\rho} \leq 16\br{\frac{p + c_2\sqrt{np} +c_2\sigma\sqrt{d np}}{np} + \gamma +2\rho} $, using the definition of $t$ in (\ref{eqn:rev_17}). We then have
\begin{align*}
&\ellod(\hat Z,Z^*)\\
&\leq 16^3c_2(1+\eta) \br{c_2\frac{1}{\sqrt{np}} +\frac{1}{n}}\br{\sqrt{\frac{\log n}{np}} + \frac{1}{\log (np)}} \br{\frac{p + c_2\sqrt{np} +c_2\sigma\sqrt{d np}}{np} + \gamma +2\rho} \\
&\quad \times \left(1+c_2'\sqrt{\frac{\log n}{n}}\right) \frac{d(d-1)\sigma^2}{2np}\\
&\quad + 3(1+\eta^{-1})\br{\frac{p + c_2\sqrt{np} +c_2\sigma\sqrt{d np}}{np}}^2\left(1+c_2'\sqrt{\frac{\log n}{n}}\right)\frac{16}{np} \frac{d^2\sigma^2}{np}  \\
&\quad + 4\gamma^{-6}\br{\frac{\sigma^2}{np}}^2 \frac{d\sigma^2}{np} + 256c_2\br{3(1+\eta^{-1}) + d\rho^{-2}} \br{\frac{2c_2}{\sqrt{np}} +\frac{2}{n\sqrt{np}}}^2 \frac{d^2\sigma^2 }{np}\\
&\quad + 64 \br{3(1+\eta^{-1})+d\rho^{-2}} \br{c_2\frac{\sigma\sqrt{d}+1}{\sqrt{np}}}^2 \frac{d^2\sigma^2}{np}.
\end{align*}
After rearrangement, there exists some constant $c_5>0$ such that
\begin{align*}
\ellod(\hat Z,Z^*)&\leq \Bigg(1+c_5\Bigg(\eta+\gamma + \rho +\sqrt{\frac{\log n}{np}} + \frac{1}{\log (np)}  + \gamma^{-6}\br{\frac{\sigma^2}{np}}^2+\sqrt{\frac{d\sigma^2}{np}} \\
&\quad + \br{\eta^{-1} + d\rho^{-2}}\br{\frac{1+d\sigma^2}{np}} \Bigg)\Bigg)\frac{d(d-1)\sigma^2}{2np}.
\end{align*}
We can take $\gamma^2 =\sqrt{d^2\sigma^2/np}$ 
(then $\frac{\gamma^2 np}{d^2\sigma^2}>c_3$ is guaranteed as long as $\frac{np}{d^2\sigma^2}>c_3^2$). We also take $\rho^2 = \sqrt{(d+d\sigma^2)/np}$ and let $\eta =\rho^2$. They are guaranteed to be smaller than $1/8$ when $\frac{np}{d}$ and $\frac{np}{d^2\sigma^2}$ are greater than some large constant.
Then,  there exists some constant $c_6>0$ such that
\begin{align*}
\ellod(\hat Z,Z^*)&\leq \Bigg(1+c_5\Bigg( \br{\frac{d+d\sigma^2}{np}}^\frac{1}{2} + \br{\frac{d^2\sigma^2}{np}}^\frac{1}{4}+ \br{\frac{d+d\sigma^2}{np}}^\frac{1}{4} +\sqrt{\frac{\log n}{np}} + \frac{1}{\log (np)}\\
&\quad + d^{-3} \br{\frac{\sigma^2}{np}}^\frac{1}{2}  +\sqrt{\frac{d\sigma^2}{np}} + (1+d)\sqrt{\frac{np}{d + d\sigma^2}}\br{\frac{1+d\sigma^2}{np}} \Bigg)\Bigg)\frac{d(d-1)\sigma^2}{2np}\\
&\leq \br{1+ c_6\br{ \br{\frac{d+d^2\sigma^2}{np}}^\frac{1}{4} + \sqrt{\frac{\log n}{np}} + \frac{1}{\log (np)}} }\frac{d(d-1)\sigma^2}{2np}.
\end{align*}
This holds with probability at least $1-n^{-9}-\ebr{-\frac{1}{32}\br{\frac{np}{\sigma^2}}^\frac{1}{4}}$.

\end{proof}

\section{Calculation for (\ref{eqn:rev_11})}\label{sec:v_calculation}
Recall the definitions of $Y^*$ and $Y$ in (\ref{eqn:rev_10}). First, we are going to show $v$, the leading eigenvector of $Y$, must be a linear combination of $e_1$ and $e_2$. Note that for any unit vector $x=(x_1,\ldots,x_n)^\T\in\mathr^n$, we have
\begin{align*}
x^\T Y x &= x^\T Y^* x + x^\T (Y-Y^*) x =\br{-\sum_{2\leq j\leq n} x_j^2 } +  \frac{\delta}{2}(x_1+x_2)^2 = -1 + x_1^2 +  \frac{\delta}{2}(x_1+x_2)^2.
\end{align*}
If $x$ maximizes the right-hand side over the unit sphere, it is obvious that neither $x_1$ nor $x_2$ can be  0. In addition, $x_1x_2 \geq 0$ and $x_1^2+x_2^2=1$ must be satisfied; otherwise the right-hand side can be made strictly larger. Then we can write $v=\alpha e_1 + \sqrt{1-\alpha^2} e_2$ where $\alpha\in[0,1]$. Since $Yv =  \frac{\delta}{2}(\alpha + \sqrt{1-\alpha^2}) e_1 + \br{\frac{\delta}{2}(\alpha + \sqrt{1-\alpha^2})  -\sqrt{1-\alpha^2}}e_2$, we have
\begin{align*}
\frac{\alpha}{ \frac{\delta}{2}(\alpha + \sqrt{1-\alpha^2})} = \frac{\sqrt{1-\alpha^2}}{\br{\frac{\delta}{2}(\alpha + \sqrt{1-\alpha^2})  -\sqrt{1-\alpha^2}}}.
\end{align*}
After rearrangement, this gives $\delta(2\alpha^2-1) =2\alpha\sqrt{1-\alpha^2}$ which means $\alpha^2>\frac{1}{2}$. Squaring it yields the equation 
$4(1+\delta^2)\alpha^4 -4(1+\delta^2)\alpha^2 +\delta^2=0$ whose solution is $\alpha^2 = \frac{1}{2}\br{1\pm \frac{1}{\sqrt{1+\delta^2}}}$.  Since $\alpha^2>\frac{1}{2}$, we have   $\alpha^2 = \frac{1}{2}\br{1+ \frac{1}{\sqrt{1+\delta^2}}}$. Hence,
\begin{align*}
v=&\sqrt{\frac{1}{2}\br{1+ \frac{1}{\sqrt{1+\delta^2}}}}e_1 + \sqrt{\frac{1}{2}\br{1- \frac{1}{\sqrt{1+\delta^2}}}} e_2.
\end{align*}
We can verify it is the eigenvector of $Y$ corresponding to the eigenvalue $\frac{1}{2}({\delta +\sqrt{1+\delta^2}-1})$.
\end{document}